\documentclass[12pt,letter]{amsart}

\usepackage{amsmath}
\usepackage{amscd}
\usepackage{amssymb,tikz-cd}
\usepackage[centering,text={15.5cm,23cm}]{geometry} 

\usepackage{amsfonts,amssymb, amscd,bm, latexsym, graphicx, psfrag, color,float}

\usepackage{caption}
\usepackage{comment}
\usepackage{wrapfig}
\usepackage[all]{xy}
\usepackage{mathrsfs}
\usepackage{mathtools}
\usepackage{marvosym}
\usepackage{stmaryrd}
\usepackage{srcltx}
\usepackage{hyperref, verbatim}
\usepackage{upgreek, MnSymbol} 
\usepackage[inline]{enumitem} 

\usepackage{braket, mathtools}

\DeclareFontFamily{U}{mathx}{}
\DeclareFontShape{U}{mathx}{m}{n}{ <-> mathx10 }{}
\DeclareSymbolFont{mathx}{U}{mathx}{m}{n}
\DeclareFontSubstitution{U}{mathx}{m}{n}
\DeclareMathAccent{\widecheck}{0}{mathx}{"71}

\definecolor{mygray}{gray}{0.75} 

\definecolor{shadecolor}{rgb}{1,0.9,0.7}

\newtheorem{theorem}{Theorem}[section]
\newtheorem{lemma}[theorem]{Lemma}
\newtheorem{lemma-definition}[theorem]{Lemma-Definition}
\newtheorem{proposition}[theorem]{Proposition}
\newtheorem{corollary}[theorem]{Corollary}
\newtheorem{conjecture}[theorem]{Conjecture}

\theoremstyle{definition}

\newtheorem{definition}[theorem]{Definition}

\theoremstyle{remark}
\newtheorem{remark}[theorem]{Remark}

\numberwithin{equation}{section}
\numberwithin{figure}{section}

\newcommand{\FF} {\mathbb{F}}

\newcommand{\bC}{\mathbb{C}}

\newcommand{\bP}{\mathbb{P}}

\newcommand{\cL}{\mathcal{L}}
\newcommand{\cM}{\mathscr{M}}

\newcommand{\cO}{\mathcal{O}}

\newcommand{\barM}{\overline{\mathcal{M}}}

\newcommand {\red}[1]{{\color{red} #1}}

\newcommand {\Spec} {\operatorname{Spec}}

\newcommand{\q} {\textbf{\emph{q}}}

\newcommand {\vdim} {\text{vdim }}

\def\mydate{\ifcase\month \or January\or February\or March\or
April\or May\or June\or July\or August\or September\or October\or 
November\or December\fi \space\number\day,\space\number\year}

\makeatletter
\let\oldcite\cite
\renewcommand{\cite}{\@ifnextchar[{\@newcite}{\oldcite}}
\def\@newcite[#1]#2{\oldcite{#2},\,#1}
\makeatother

\begin{document}

\title[Higher Genus Gromov-Witten Invariants on Smooth Log Calabi-Yau pairs]
{Higher Genus Gromov-Witten Invariants from Projective Bundles on Smooth Log Calabi-Yau pairs}

\author{Benjamin Zhou} 
\address{\tiny Benjamin Zhou, Yau Mathematical Sciences Center, Tsinghua University,
Beijing, China}
\email{byzhou@mail.tsinghua.edu.cn}

\begin{abstract}
    Let $(X,E)$ be a smooth log Calabi-Yau pair consisting of a smooth Fano surface $X$ and a smooth anti-canonical divisor $E$. We obtain certain higher genus local Gromov-Witten invariants from the projectivization of the canonical bundle $Z := \bP(K_X \oplus \cO_X)$, using the degeneration formula for stable log maps \cite{KLR}. We evaluate an invariant in the degeneration using the relationship between $q$-refined tropical curve counting and logarithmic Gromov-Witten theory with $\lambda_g$-insertion \cite{Bou}. As a corollary, we use flops to prove a blow up formula for higher genus invariants of $Z$. Additionally assuming $X$ is toric, we prove an all-genus correspondence between open invariants of an outer Aganagic-Vafa brane $L \subset K_X$ and closed invariants of $Z$ that generalizes a genus-0 open-closed equality of \cite{Cha} to all-genus, by using an argument in \cite{GRZZ}. 
\end{abstract}

\maketitle
\setcounter{tocdepth}{1}
\tableofcontents


\section{Introduction}

\label{sec:intro}

A log Calabi-Yau pair $(X,E)$ consists of a smooth complex projective surface with a possibly singular anti-canonical divisor $E \in |-K_X|$. When $E = E_1 + \ldots E_l$ for $l \geq 1$ is an anti-canonical nodal curve, the pairs $(X,E)$ are known as \textit{Looijenga pairs}, which are shown to have a rich enumerative theory \cite{BBvG}. A smooth log Calabi-Yau pair $(X, E)$ additionally requires $l = 1$; by adjunction, the anti-canonical divisor $E$ is a smooth elliptic curve. For smooth log Calabi-Yau pairs, the logarithmic Gromov-Witten theory of $X$ with multiple contact orders along $E$ is equated with tropical curve counting \cite{Gra}, as well as with open Gromov-Witten invariants of an outer Aganagic-Vafa brane in $K_X$ in all-genus \cite{GRZ} \cite{GRZZ}. 

Log Calabi-Yau pairs also provide 2-dimensional examples of Gross-Siebert mirror symmetry \cite{GHK} \cite{GS16}. Roughly, the Gross-Siebert program associates to $(X, E)$ a mirror dual $\check{X}$, which is a quasiprojective variety obtained by gluing together cluster torus charts via wall crossing transformations. The mirror dual $\check{X}$ is given by the spectrum of an algebra $A$ defined by \textit{theta functions} $\theta_p: \check{X} \rightarrow \mathbb{C}$ which count \textit{broken lines}, or the tropical analogue of holomorphic discs, in $\check{X}$ \cite{GS16} \cite{GHK}. Let $P \subset H_2(X, \mathbb{Z})$ be the monoid of effective curve classes. For smooth log Calabi-Yau pairs $(X, E)$, $A$ is an $\mathbb{N}$-graded algebra indexed by theta functions $\theta_p$ for $p \in \mathbb{N}$,

\[
A = \bigoplus_{p \in \mathbb{N}} \mathbb{C}[P] \cdot \theta_p
\]
The multiplication rule $\theta_p \cdot \theta_q = \sum_r N^{\beta}_{pqr} \theta_r$, is given by the structure constants $N^{\beta}_{pqr}$ and determined by \textit{punctured Gromov-Witten invariants} in curve class $\beta$ of contact orders $p, q$ and $-r$ to $E$ \cite{GS16}. The invariants $N^{\beta}_{pqr}$ are in turn expressed by two-pointed, genus-0 logarithmic invariants,

\begin{equation}
\label{eq:punctured_log}
    N^{\beta}_{pqr} = (p-r) R_{0, (q, p-r)} + (q-r) R_{0, (p, q-r)}
\end{equation}
where $R_{0, (a, b)}$ denotes the two-pointed, genus-0 logarithmic invariant of $(X, E)$ with one fixed contact point of order $a$ and one varying contact point of order $b$. Equation \ref{eq:punctured_log} is proven in Theorem 1.1, \cite{Wan} via analyzing moduli spaces of punctured stable maps, as well as in Proposition 5.2, \cite{GRZ} using a tropical/holomorphic correspondence for logarithmic invariants \cite{Gra}. Hence, the genus-0, two-pointed log invariants of $(X, E)$ express the structure constants of the algebra $A$ of theta functions $\theta_p$.

By introducing a formal parameter $\q$, the \textit{quantum theta algebra} $A(\q)$ has a basis given by \textit{quantum theta functions} $\theta_p(\q)$ with multiplication rule in the quantum torus $xy = \q yx$. When $E$ is smooth, the quantum theta algebra $A(\q)$ is $\mathbb{N}$-graded and given by,

\begin{equation}
\label{eq:qtheta_alg}
    A(\q) = \bigoplus_{p \in \mathbb{N}} \mathbb{C}[P] \cdot \theta_p(\q)
\end{equation}
It is shown in Theorem 4.14, \cite{GRZ}, that the quantum theta functions $\theta_p(\q)$ are given by,

\[
\theta_p(\q) = y^q + \sum_{p\geq 1}\sum_{\beta | \beta\cdot E = p+q} R^{trop}_{g, (p,q)}(X, \beta) \hbar^{2g} Q^{\beta} t^{\deg \beta} y^{-p}
\]
where $R^{trop}_{g, (p,q)}(X, \beta)$ is the count of genus-$g$ tropical curves  in class $\beta$ with two unbounded legs intersecting the elliptic curve $E$ at two points of order $p$ and $q$. The variable $y$ is a cluster monomial on $\check{X}$ given by the unique unbounded direction in the dual intersection complex associated to a smooth log Calabi-Yau pair $(X, E)$. In Propositions 5.1, 5.2 of \cite{GRZ}, the multiplication rule for quantum theta functions $\theta_p(\q)$ is expressed by $R^{trop}_{g, (p,q)}(X, \beta)$, which in turn are related to higher genus two-pointed log invariants of $(X, E)$ with $\lambda_g$-insertion by \cite{Gra}. Setting $\q = 1$ recovers the genus-0 structure constants of Equation \ref{eq:punctured_log}. For more results on quantum theta functions, we refer to \cite{GRZ} \cite{GRZZ} \cite{Man2}.

\subsection{Main results}
\label{sec:main_results}
In this paper, we relate higher genus invariants of projective bundles on smooth log Calabi-Yau pairs $(X,E)$ to higher genus two-pointed log Gromov-Witten invariants of $(X, E)$ with $\lambda_g$-insertion. This provides a new way of expressing the structure constants of the quantum theta algebra $A(\q)$ (\ref{eq:qtheta_alg}). In addition, we show an alternative way to obtain certain higher genus closed and open Gromov-Witten invariants of Calabi-Yau 3-folds. 

Denote $X(\log E)$ to be the log scheme $X$ with divisorial log structure given by $E$. Define $Z:= \bP(K_X \oplus \cO_X)$ to be the projective compactification of the canonical bundle $K_X$. The effective curve classes of $Z$ decompose as $NE(Z) = i_*NE(X) \oplus NE(\bP^1)$, where $i: X \hookrightarrow Z$ is the inclusion of $X$ as the zero-section of $Z$, and $\bP^1$ is a fiber of $Z.$ Take $\beta+h \in NE(Z)$, where $\beta \in NE(X)$ and $h$ is a generator of $NE(\bP^1).$ Consider the moduli space $\barM_{g,1}(Z, \beta+h)$ of genus-$g$, 1-pointed maps to $Z$ in curve class $\beta+h$, which has virtual dimension,

\[
\vdim \barM_{g,1}(Z, \beta+h) = (\dim Z - 3)(1-g) + \int_{\beta+h} c_1(TZ) + 1
\]
As $c_1(TZ)(\beta) = 0$ and $\dim Z = 3$, we have $\text{vdim }\barM_{g,1}(Z, \beta+h)= 3$. This leads us to define the following closed Gromov-Witten invariant,

\begin{equation}
\label{eq:def_of_invariant}
 N_{g,1}(Z, \beta+h) := \int_{[\overline{\mathcal{M}}_{g,1}(Z, \beta+h)]^{vir}} ev^*[pt]  
\end{equation}
where $[pt] \in H^6(Z, \mathbb{Z})$ is the Poincar\'e dual of a point in $Z$. The quantity $N_{g,1}(Z, \beta+h)$ is a virtual count of genus-$g$ curves in $Z$ passing through a single point. 

\subsubsection{Higher genus correspondence for projective bundles}
Let $R_{g,(1,\beta\cdot E - 1)}(X(\log E), \beta)$ to be the genus-$g$, logarithmic Gromov-Witten invariant of $X$ in curve class $\beta$ counting curves intersecting $E$ at a prescribed point with contact order $1$ and a non-prescribed point with contact order $\beta\cdot E - 1$ (see \ref{eq:vertexA_invariant} for a definition). Let $\pi: \widehat{X} \rightarrow X$ be the blow up of $X$ at a point with exceptional curve $C$, and let $n_g(K_{\widehat{X}}, \pi^*\beta-C)$ be the genus-$g$, Gopakumar-Vafa invariant of $K_{\widehat{X}}$ in class $\pi^*\beta- C$ defined by multiple cover formulas \cite{GV1} \cite{GV2} for the corresponding local Gromov-Witten invariant $N_{g,0}(K_{\widehat{X}}, \pi^*\beta-C)$. In Section \ref{sec:proof-of-main}, we relate the generating function of $N_{g,1}(Z, \beta+h)$ to the $n_g(K_{\widehat{X}}, \pi^*\beta-C)$ for all $g$ and $\beta \in NE(X)$,

\begin{theorem}[= Theorem \ref{thm:main}]
\label{thm:main_intro}
There exists constants $c(g, \beta) \in \mathbb{Q}$ (described explicitly in Equation \ref{eq:proof-of-oP1,5}) such that,
\begin{multline*}
    \sum_{\substack{g \geq 0, \\ \beta \in NE(X)}}N_{g,1}(Z, \beta+h)\hbar^{2g}Q^{\beta} = \\ \sum_{\substack{g \geq 0, \\ \beta \in NE(X)}}\left[c(g, \beta)n_{g}\left(K_{\widehat{X}}, \pi^*\beta-C\right)\left(2\sin\frac{\hbar}{2}\right)^{2g - 2}Q^{\beta}\right] - \Delta
\end{multline*}
where the discrepancy $\Delta$ (Equation \ref{eq:delta_pl}) is expressed by the Gromov-Witten theory of $E$, and genus-$g$, 2-pointed logarithmic invariants $R_{g, (1,\beta\cdot E - 1)}(X(\log E), \beta)$ for all $g \geq 0$ and $\beta \in NE(X)$.
\end{theorem}
The proof of Theorem \ref{thm:main_intro} comes from extending arguments in \cite{vGGR} \cite{Wan} to higher genus when $\dim X = 2$, and using the degeneration formula for stable log maps \cite{KLR} and the higher genus log-local principle \cite{BFGW}. The constants $c(g, \beta)$ can be explicitly determined (Equation \ref{eq:proof-of-oP1,5}). In Proposition \ref{prop:vertexC}, we show how an invariant arising from the degeneration can be computed using the relationship between refined tropical curve counting and logarithmic Gromov-Witten invariants of toric surfaces with $\lambda_g$-insertion \cite{Bou}. 

\subsubsection{Blow up formula for projective bundles}
Define $W := Bl_p Z$ to be the blow up of $Z$ at a point $p$ along its infinity section. Let $N_{g,0}(W, \beta+\tilde{L})$ be the genus-g, unmarked closed Gromov-Witten invariant in class $\beta+\tilde{L}$, where $\tilde{L}$ is the strict transform of the fiber of $Z$ passing through $p$. In Section \ref{sec:blow-up}, we prove a blow up formula relating the generating function of $N_{g,1}(Z, \beta+h)$ to $N_{g,0}(W, \beta+\tilde{L})$,

\begin{theorem}[= Theorem \ref{thm:blow-up}]
\label{thm:main_blow-up}
    Let $W := Bl_p Z$ be the blow up of $Z$ at a point $p$ on its infinity section. 
    
    \[
    \sum_{\substack{g \geq 0, \\ \beta \in NE(X)}}N_{g,1}(Z, \beta+h)\hbar^{2g}Q^{\beta} =\sum_{\substack{g \geq 0, \\ \beta \in NE(X)}}\left[c(g, \beta)N_{g,0}(W, \beta + \tilde{L})\hbar^{2g}Q^{\beta}\right] - \Delta
    \]
    where $c(g,\beta) \in \mathbb{Q}$ and the $\Delta$ are given in Theorem \ref{thm:main_intro}. 
\end{theorem}
The proof of Theorem \ref{thm:main_blow-up} follows from Theorem \ref{thm:main_intro} and the invariance of Gromov-Witten invariants under flops of 3-folds \cite{LR}.

\subsubsection{Open-closed correspondence for projective bundles}

For this section, suppose $X$ is additionally toric, and $\pi: \widehat{X} \rightarrow X$ is a toric blow up at a point. Let $L \subset K_X$ be an outer Aganagic-Vafa (AV) brane (see Section 2.4, \cite{FL} for more details on AV-branes). Defined using stable relative maps \cite{FL}, let $O_g(K_X/L, \beta+\beta_0, 1)$ be the genus-$g$, 1-holed, open Gromov-Witten invariant with boundary on $L$ in framing-0, winding-1 and curve class $\beta+\beta_0 \in H_2(K_X, L)$, where $\beta_0 \in H_2(K_X, L)$ is a relative homology class with boundary on $L$. Let $n^{open}_{g}(K_X/L, \beta+\beta_0, 1)$ be the open-BPS invariant corresponding to $O_g(K_X/L, \beta+\beta_0, 1)$ defined by multiple cover formulas in \cite{MV}. In Section \ref{sec:op}, we use Theorem \ref{thm:main_intro} to prove an open-closed correspondence for projective bundles in all-genus,

\begin{theorem}[= Theorem \ref{thm:op}]
\label{thm:intro_op}
    Suppose that $X$ is a toric Fano surface, $Z = \bP(K_X \oplus \cO_X)$, and $K_X$ the toric canonical bundle. 
    
    \begin{multline*}
    \sum_{\substack{g \geq 0, \\ \beta \in NE(X)}} N_{g,1}(Z, \beta+h)\hbar^{2g}Q^{\beta} = \\ \sum_{\substack{g \geq 0, \\ \beta \in NE(X)}}\left[(-1)^{g+1}c(g, \beta)n^{open}_{g}(K_X/L, \beta+\beta_0, 1)\left(2\sin\frac{\hbar}{2}\right)^{2g - 2}Q^{\beta}\right] - \Delta
    \end{multline*}
    where $c(g,\beta) \in \mathbb{Q}$ and $\Delta$ are as in Theorem \ref{thm:main_intro}.
\end{theorem}

As Gromov-Witten invariants and BPS invariants are uniquely determined from each other by multiple cover formulas, the $n^{open}_g(K_X/L, \beta+\beta_0, 1)$ can be defined by the $O_g(K_X/L, \beta+\beta_0, 1)$. Theorem \ref{thm:intro_op} implies that the higher genus open invariants $O_g(K_X/L, \beta+\beta_0, 1)$ are expressed by genus-$g$, 1-pointed, closed invariants $N_{g,1}(Z, \beta+h)$ and $\Delta$, which is expressible by the stationary Gromov-Witten theory of $E$, and genus-$g$, 2-pointed log invariants $R_{g, (1,\beta\cdot E - 1)}(X(\log E), \beta)$. The proof of Theorem \ref{thm:intro_op} relies on an equality of open and closed BPS invariants for toric Calabi-Yau threefolds shown in Corollary 5.5 of \cite{GRZZ}, which we state in Corollary \ref{cor:winding-1}. Theorem \ref{thm:intro_op} extends the genus-0 equality of open and closed Gromov-Witten invariants \cite{Cha} to all-genus. 

We give explicit genus-1 formulas for Theorems \ref{thm:main_intro}, \ref{thm:main_blow-up} in Corollaries \ref{cor:maing1}, \ref{cor:blow-up12}  respectively, using formulas from the Appendix. 


\begin{remark}
\label{rem:g=0}
The invariant $N_{0,1}(Z, \beta+h)$ was equated with the genus-0, 1-holed, open Gromov-Witten invariant of a moment torus fiber in the toric Calabi-Yau 3-fold $K_X$ in \cite{Cha}. By the remarks made in Section 2.2, \cite{GRZ}, the genus-0, open invariant of moment fiber of $K_X$ is equal to the genus-0, open invariant $O_0(K_X/L, \beta+\beta_0, 1)$ of an outer AV-brane. In \cite{Wan}, $N_{0,1}(Z)$ is related to the genus-0, 2-pointed, logarithmic Gromov-Witten invariant $R_{0,2}(X(\log E), \beta)$. A genus-0 correspondence between $O_0(K_X/L, \beta+\beta_0, 1)$ and $R_{0,2}(X(\log E), \beta)$ is made in \cite{GRZ}. Hence in genus-0, these results suggest equalities (that are up to rational constants depending on $\beta$) between $N_{0,1}(Z, \beta+h)$, $O_0(K_X/L, \beta+\beta_0, 1)$, and $R_{0,2}(X(\log E), \beta)$. These results are summarized in Figure \ref{fig:g=0_triangle}. We have that Theorem \ref{thm:main_intro}, Theorem 5.13 of \cite{GRZZ}, and Theorem \ref{thm:intro_op} extends these genus-0 results to higher genus. 

\begin{figure}[h!]
    \centering
    \begin{tikzcd}
    O_0(K_X/L, \beta+\beta_0, 1) \arrow[rr, equal]  && N_{0,1}(Z, \beta+h) \arrow[dl, equal]\\
    & R_{0,2}(X(\log E), \beta) \arrow[ul, equal]
    \end{tikzcd}
    \caption{Genus-0 equalities, holding up to explicitly determined rational constants depending on $\beta \in NE(X)$, between open, log, and closed invariants associated to $(X,E).$ See Remark \ref{rem:g=0}.}
    \label{fig:g=0_triangle}
\end{figure}

\end{remark}

\begin{remark}
    This work shares some similarities with results in \cite{GRZZ}. Both this work and Theorem 5.2, \cite{GRZZ} relate higher genus 2-pointed log invariants $R_{g,2}(X(\log E), \beta)$ to Gopakumar-Vafa invariants of $K_{\widehat{X}}$. The main difference between the two works is in the way $R_{g,2}(X(\log E), \beta)$ is obtained; here we use the degeneration formula to obtain $R_{g,2}(X(\log E), \beta)$ from higher genus invariants $N_{g,1}(Z, \beta+h)$ of the projective bundle, while \cite{GRZZ} uses the scattering diagram of $X(\log E)$ defined from Gross-Siebert mirror symmetry and computes tropical curves with 2 unbounded legs in the dual intersection complex of $X(\log E)$. Corollary \ref{cor:winding-1} is also used to prove a higher genus open-log correspondence in Theorem 5.13, \cite{GRZZ}; here we use it to show a higher genus open-closed correspondence for projective bundles. The discrepancy $\Delta$ in Theorem \ref{thm:main_intro} is obtained in a similar procedure as the $\Delta$ in Theorem 5.2, \cite{GRZZ}, however we note that the two discrepancies are not equal.
\end{remark}

\begin{remark}
Projective bundles appear in Gromov-Witten theory, as "bubble components" in expanded degenerations of stable relative maps \cite{Li}. It was shown in \cite{Fan} that if two vector spaces $V_1$ and $V_2$ have the same Chern classes, then the Gromov-Witten theory of their projectivizations $\bP(V_i)$ are equal. In \cite{CGT}, it is shown that the Virasoro constraints are satisfied for toric projective bundles if and only if they are satisfied for the base. 
\end{remark}

\subsection{Notation}
Denote $A_d(X)$ to be the Chow group of $d$-dimensional cycles of $X$, and $A^d(X)$ the group of codimension-$d$ cycles. Denote $NE(X)$ to be the effective curve classes of $X$. We  write $[pt] \in H^{top}(X)$ to be the Poincar\'e dual of a point in $X$. If $X$ is a toric variety, we write $\partial X$ to be its toric boundary. We will often notationally suppress the curve class or log structure and write $N_{g,1}(Z, \beta+h)$ as $N_{g,1}(Z)$, or $R_{g,(p,q)}(X(\log E), \beta)$ as $R_{g,(p,q)}(X)$. We shall use formal variables $Q^{\beta}$ to label effective curve classes, and $\hbar^{2g}$ to label the genus. Let $q = e^{i\hbar}$.

\subsection{Outline of the paper}
In Section \ref{sec:deg}, we describe the degeneration argument and the computation of higher genus invariants. In Section \ref{sec:proof-of-main}, we prove Theorem \ref{thm:main_intro}. In Section \ref{sec:blow-up}, we prove Theorem \ref{thm:main_blow-up}. In Section \ref{sec:op}, we prove Theorem \ref{thm:intro_op}. 

\subsection{Acknowledgements}
This work grew out of the author's PhD thesis at Northwestern University. I am very grateful to my PhD advisor Eric Zaslow, as well as to Helge Ruddat, Tim Gr\"afnitz, Pierrick Bousseau, Yu Wang, and Mark Gross for helpful conversations and suggestions.

\section{Degeneration of projective bundles}
\label{sec:deg}

We use the degeneration formula for stable log maps \cite{KLR} to compute the invariants $N_{g,1}(Z, \beta+h)$. When $g = 0$, $N_{0,1}(Z, \beta+h)$ was computed in \cite{Wan}. We explain how some of the arguments in \cite{Wan} \cite{vGGR} can be extended to higher genus when $\dim X = 2$.

\subsection{Degeneration}
\label{subsec:deg}

We take a degeneration of $Z$ (see Section 4, \cite{Wan}) to a normal crossings singular fiber. Let $\mathscr{X} = Bl_{E\times 0}(X\times \mathbb{A}^1)$ be the deformation to the normal cone $\pi: \mathscr{X}  \rightarrow \mathbb{A}^1$. The fiber $\mathscr{X}^{-1}(t)$ when $t \neq 0$ is isomorphic to $X$. The special fiber $\mathscr{X}^{-1}(0)$ is isomorphic to $X \sqcup_E \mathbb{P}(N_{E/X}\oplus \mathcal{O}_E)$. Denote $Y$ to be the exceptional hypersurface $\mathbb{P}(N_{E/X}\oplus \mathcal{O}_E)$. Let $E_0$ and $E_{\infty}$ be the sections of $Y$ corresponding to the summands $N_{E/X}$ and $\mathcal{O}_E$ respectively. Let $\pi_Y: Y \rightarrow E_0$ be the projection map. We will at times write $\pi_Y$ as $Y \rightarrow E$ when there is no confusion.

Let $\mathscr{E} := \overline{\pi^{-1}(E\times\mathbb{A}^1 \setminus E \times 0)}$ be the strict transform of $E \times \mathbb{A}^1$. Define the space $\mathcal{L} = \mathbb{P}(\mathcal{O}_{\mathscr{X}}(-\mathscr{E}) \oplus \mathcal{O}_{\mathscr{X}})$ on $\mathscr{X}$, which serves as a degeneration of $Z$, as the generic fiber $\mathcal{L}_t$ for $t \neq 0$ is isomorphic to $Z$, and the special fiber $\cL_0$ is isomorphic to $X\times\mathbb{P}^1 \sqcup_{E\times\mathbb{P}^1} \mathbb{P}(\mathcal{O}_Y(-E_{\infty}) \oplus \mathcal{O}_Y).$ Denote $\cL_X := X \times \bP^1, \cL_E := E \times \bP^1$ and $\cL_Y := \mathbb{P}(\mathcal{O}_Y(-E_{\infty}) \oplus \mathcal{O}_Y)$, hence $\cL_0 = \cL_X \sqcup_{\cL_E} \cL_Y$. We have projection maps $\pi_{\cL}: \cL \rightarrow \mathbb{A}^1$, $\pi_{\cL_Y}: \cL_Y \rightarrow Y$, and $\pi_{\cL_X}: \cL_X \rightarrow X$. The restriction of $\cL_Y$ onto a fiber of $Y \rightarrow E$ is the first Hirzebruch surface $\FF_1 = \bP(\cO_{\bP^1}(-1) \oplus \cO_{\bP^1})$.

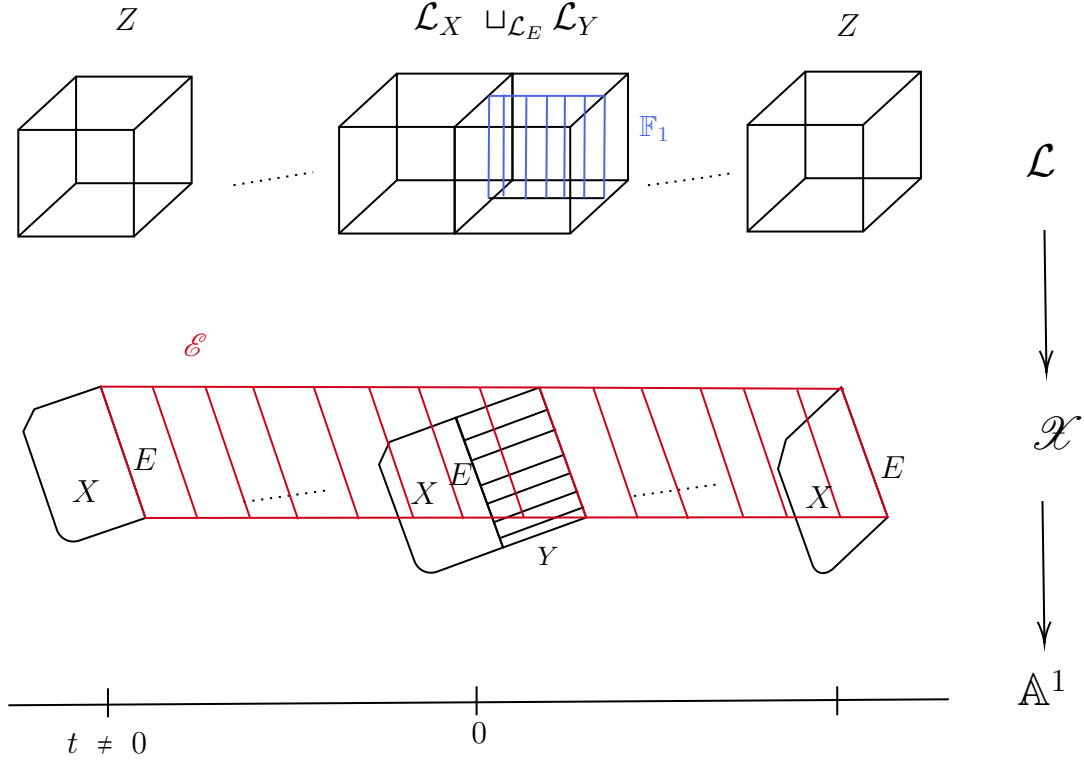
\begin{figure}[h!]
    \centering

\tikzset{every picture/.style={line width=0.75pt}} 

\begin{tikzpicture}[x=0.75pt,y=0.75pt,yscale=-1,xscale=1]

\draw    (79,457.5) -- (551,454.5) ;
\draw   (114.5,374.86) .. controls (109.9,376.43) and (104.9,373.97) .. (103.33,369.37) -- (86.53,320.12) -- (92.02,308.95) -- (125.33,297.59) -- (147.81,363.5) -- cycle ;
\draw   (294.16,390.45) .. controls (289.58,392.13) and (284.52,389.78) .. (282.84,385.21) -- (265.06,336.68) -- (270.31,325.36) -- (303.44,313.22) -- (327.29,378.32) -- cycle ;
\draw   (493.18,388.71) .. controls (489.04,392.61) and (484.26,391.8) .. (482.51,386.89) -- (465.3,338.89) -- (469.28,324.14) -- (497,298) -- (520.33,363.1) -- cycle ;
\draw   (327.29,378.32) -- (303.74,313.12) -- (345.45,298.05) -- (369,363.25) -- cycle ;
\draw  [dash pattern={on 0.84pt off 2.51pt}]  (197,355.5) -- (240,349.5) ;
\draw  [dash pattern={on 0.84pt off 2.51pt}]  (393,352.5) -- (436,346.5) ;
\draw   (245,167.25) -- (303,167.25) -- (303,220.75) -- (245,220.75) -- cycle ;
\draw   (274,140.5) -- (332,140.5) -- (332,194) -- (274,194) -- cycle ;
\draw   (303,167.25) -- (361,167.25) -- (361,220.75) -- (303,220.75) -- cycle ;
\draw   (332,140.5) -- (390,140.5) -- (390,194) -- (332,194) -- cycle ;
\draw   (450,166.25) -- (508,166.25) -- (508,219.75) -- (450,219.75) -- cycle ;
\draw   (479,139.5) -- (537,139.5) -- (537,193) -- (479,193) -- cycle ;
\draw    (245,167.25) -- (274,140.5) ;
\draw    (303,167.25) -- (332,140.5) ;
\draw    (245,220.75) -- (274,194) ;
\draw    (303,220.75) -- (332,194) ;
\draw    (361,167.25) -- (390,140.5) ;
\draw    (361,220.75) -- (390,194) ;
\draw    (450,166.25) -- (479,139.5) ;
\draw    (508,166.25) -- (537,139.5) ;
\draw    (450,219.75) -- (479,193) ;
\draw    (508,219.75) -- (537,193) ;
\draw    (598,355) -- (598.97,424.5) ;
\draw [shift={(599,426.5)}, rotate = 269.2] [color={rgb, 255:red, 0; green, 0; blue, 0 }  ][line width=0.75]    (10.93,-3.29) .. controls (6.95,-1.4) and (3.31,-0.3) .. (0,0) .. controls (3.31,0.3) and (6.95,1.4) .. (10.93,3.29)   ;
\draw    (599,219) -- (599.97,288.5) ;
\draw [shift={(600,290.5)}, rotate = 269.2] [color={rgb, 255:red, 0; green, 0; blue, 0 }  ][line width=0.75]    (10.93,-3.29) .. controls (6.95,-1.4) and (3.31,-0.3) .. (0,0) .. controls (3.31,0.3) and (6.95,1.4) .. (10.93,3.29)   ;
\draw  [dash pattern={on 0.84pt off 2.51pt}]  (400,196.5) -- (426,193) -- (443,190.33) ;
\draw  [dash pattern={on 0.84pt off 2.51pt}]  (192,196.5) -- (232,190) ;
\draw   (84,168.75) -- (142,168.75) -- (142,222.25) -- (84,222.25) -- cycle ;
\draw   (113,142) -- (171,142) -- (171,195.5) -- (113,195.5) -- cycle ;
\draw    (84,168.75) -- (91.23,162.08) -- (113,142) ;
\draw    (142,168.75) -- (171,142) ;
\draw    (84,222.25) -- (113,195.5) ;
\draw    (142,222.25) -- (171,195.5) ;
\draw    (312,334.5) -- (354,319) ;
\draw    (316.23,347.16) -- (358.23,331.66) ;
\draw    (319,356.5) -- (361,341) ;
\draw    (323,365.5) -- (365,350) ;
\draw    (308,324.5) -- (350,309) ;
\draw    (326,373.5) -- (368,358) ;
\draw    (314,448) -- (314,462.5) ;
\draw    (129,449) -- (129,463.5) ;
\draw    (495,448) -- (495,462.5) ;
\draw    (320,203) -- (378,203) ;
\draw [color={rgb, 255:red, 74; green, 104; blue, 226 }  ,draw opacity=1 ]   (320,203) -- (320.33,151.67) ;
\draw [color={rgb, 255:red, 74; green, 104; blue, 226 }  ,draw opacity=1 ]   (320.33,151.67) -- (378.33,151.67) ;
\draw [color={rgb, 255:red, 74; green, 104; blue, 226 }  ,draw opacity=1 ]   (378,203) -- (378.33,151.67) ;
\draw [color={rgb, 255:red, 74; green, 104; blue, 226 }  ,draw opacity=1 ]   (327.5,202) -- (327.67,151) ;
\draw [color={rgb, 255:red, 74; green, 104; blue, 226 }  ,draw opacity=1 ]   (338.67,203.67) -- (339,152.33) ;
\draw [color={rgb, 255:red, 74; green, 104; blue, 226 }  ,draw opacity=1 ]   (349,203) -- (349.33,151.67) ;
\draw [color={rgb, 255:red, 74; green, 104; blue, 226 }  ,draw opacity=1 ]   (368,203) -- (368.33,151.67) ;
\draw [color={rgb, 255:red, 74; green, 104; blue, 226 }  ,draw opacity=1 ]   (358,203) -- (358.33,151.67) ;
\draw [color={rgb, 255:red, 208; green, 2; blue, 27 }  ,draw opacity=1 ]   (125.33,297.59) -- (345.58,298) ;
\draw [color={rgb, 255:red, 208; green, 2; blue, 27 }  ,draw opacity=1 ]   (147.81,363.5) -- (369,363.25) ;
\draw [color={rgb, 255:red, 208; green, 2; blue, 27 }  ,draw opacity=1 ]   (369,363.25) -- (520.33,363.1) ;
\draw [color={rgb, 255:red, 208; green, 2; blue, 27 }  ,draw opacity=1 ]   (345.58,298) -- (498.4,298.69) ;
\draw [color={rgb, 255:red, 208; green, 2; blue, 27 }  ,draw opacity=1 ]   (147.81,363.5) -- (125.33,297.59) ;
\draw [color={rgb, 255:red, 208; green, 2; blue, 27 }  ,draw opacity=1 ]   (520.33,363.1) -- (497,298) ;
\draw [color={rgb, 255:red, 208; green, 2; blue, 27 }  ,draw opacity=1 ]   (173.87,363.6) -- (151.4,297.69) ;
\draw [color={rgb, 255:red, 208; green, 2; blue, 27 }  ,draw opacity=1 ]   (200.37,363.6) -- (177.9,297.69) ;
\draw [color={rgb, 255:red, 208; green, 2; blue, 27 }  ,draw opacity=1 ]   (224.37,364.1) -- (201.9,298.19) ;
\draw [color={rgb, 255:red, 208; green, 2; blue, 27 }  ,draw opacity=1 ]   (254.87,364.1) -- (232.4,298.19) ;
\draw [color={rgb, 255:red, 208; green, 2; blue, 27 }  ,draw opacity=1 ]   (282.37,363.6) -- (259.9,297.69) ;
\draw [color={rgb, 255:red, 208; green, 2; blue, 27 }  ,draw opacity=1 ]   (307.4,364) -- (284.9,298.19) ;
\draw [color={rgb, 255:red, 208; green, 2; blue, 27 }  ,draw opacity=1 ]   (337.87,363.6) -- (315.4,297.69) ;
\draw [color={rgb, 255:red, 208; green, 2; blue, 27 }  ,draw opacity=1 ]   (394.87,363.6) -- (372.4,297.69) ;
\draw [color={rgb, 255:red, 208; green, 2; blue, 27 }  ,draw opacity=1 ]   (369,363.25) -- (345.58,298) ;
\draw [color={rgb, 255:red, 208; green, 2; blue, 27 }  ,draw opacity=1 ]   (420.2,364) -- (396.58,298) ;
\draw [color={rgb, 255:red, 208; green, 2; blue, 27 }  ,draw opacity=1 ]   (447.8,362.8) -- (426.4,298.69) ;
\draw [color={rgb, 255:red, 208; green, 2; blue, 27 }  ,draw opacity=1 ]   (469.8,362.8) -- (447.9,298.19) ;
\draw [color={rgb, 255:red, 208; green, 2; blue, 27 }  ,draw opacity=1 ]   (496.6,363.2) -- (474.9,299.19) ;

\draw (109.6,343) node [anchor=north west][inner sep=0.75pt]    {$X$};
\draw (139.8,327.13) node [anchor=north west][inner sep=0.75pt]    {$E$};
\draw (477.53,345.93) node [anchor=north west][inner sep=0.75pt]    {$X$};
\draw (515.13,331.2) node [anchor=north west][inner sep=0.75pt]    {$E$};
\draw (298.15,334.5) node [anchor=north west][inner sep=0.75pt]    {$E$};
\draw (279.37,344) node [anchor=north west][inner sep=0.75pt]    {$X$};
\draw (493,109.4) node [anchor=north west][inner sep=0.75pt]  [font=\normalsize]  {$Z$};
\draw (343,376.4) node [anchor=north west][inner sep=0.75pt]  [font=\footnotesize]  {$Y$};
\draw (256,103.4) node [anchor=north west][inner sep=0.75pt]  [font=\large]  {$\ \ \ \ \mathcal{L}_{X} \ \sqcup _{\mathcal{L}_{E}}\mathcal{L}_{Y}$};
\draw (588,174.4) node [anchor=north west][inner sep=0.75pt]  [font=\Large]  {$\mathcal{L}$};
\draw (590,311.4) node [anchor=north west][inner sep=0.75pt]  [font=\Large]  {$\mathscr{X} \ $};
\draw (584,438.4) node [anchor=north west][inner sep=0.75pt]  [font=\Large]  {$\mathbb{A}^{1}$};
\draw (310,464.4) node [anchor=north west][inner sep=0.75pt]    {$0$};
\draw (107,469.4) node [anchor=north west][inner sep=0.75pt]    {$t\ \neq \ 0$};
\draw (394.33,160.4) node [anchor=north west][inner sep=0.75pt]  [font=\small]  {$\mathbb{\textcolor[rgb]{0.29,0.41,0.89}{F}}\textcolor[rgb]{0.29,0.41,0.89}{_{1}}$};
\draw (165.33,269.73) node [anchor=north west][inner sep=0.75pt]  [color={rgb, 255:red, 208; green, 2; blue, 27 }  ,opacity=1 ]  {$\mathscr{E}$};
\draw (131,106.4) node [anchor=north west][inner sep=0.75pt]  [font=\normalsize]  {$Z$};

\end{tikzpicture}

    \caption{The degeneration $\mathcal{L} \rightarrow \mathbb{A}^1$ of $Z$ formed by the divisor $\mathscr{E}$ (red). Restricting $\cL_Y$ over a fiber of $Y \rightarrow E$, we have the first Hirzebruch surface $\FF_1 = \bP(\cO_{\bP^1}(-1) \oplus \cO_{\bP^1})$ (blue).}
    \label{fig:deg}
\end{figure}

\subsection{Stable log maps to $\cL$}

We consider stable log maps to the degeneration $\cL \rightarrow \mathbb{A}^1$ of $Z$. We take the divisorial log structure on $\cL$ given by the central fiber $\cL_0$. Let $\barM_{g,n+r}(\cL(\log \cL_0), \beta+h)$ be the moduli space of genus-$g$, basic stable log maps to $\cL$ in the curve class $\beta+h$, with $n$ interior marked points and $r$ relative marked points.

\begin{remark}
\label{rem:deg_curve_class}
    We denote the curve class of stable log maps in $\barM(\cL, \beta+h)$ by $\beta+h$, which has the following meaning: recall that $\beta \in NE(X)$ and $h \in NE(\bP^1)$ is the class of a $\bP^1$-fiber. The class $\beta+h$ lives in $H_*(\cL_t) \cong H_*(Z)$ for $t \neq 0$. When writing $\beta+h$ as a curve class in $\cL$, we mean a global lifting of $\beta+h$ to some class $\alpha \in H_*(\cL)$ satisfying $\alpha|_{\cL_t} = \beta+h$ for all $t \neq 0.$ On the central fiber $\cL_0 = \cL_X \sqcup_{\cL_E} \cL_Y$, if we decompose $\alpha|_{\cL_0} = \beta_X + \beta_Y$ with $\beta_X \in H_*(\cL_X), \beta_Y \in H_*(\cL_Y)$, then $\alpha|_{\cL_0}$ satisfies $(\pi_{\cL_X})_*\beta_X + (\pi_Y)_*\beta_Y = \beta+h$ in the formalism of the degeneration formula \cite{Li}. Applied to curve classes $\beta_Y$ in $\cL_Y$, the map $\pi_Y$ contracts fibers of $Y \rightarrow E.$ For simplicity, we write maps to $\cL$ in class $\alpha$ as maps in class $\beta+h$.
\end{remark}

Stable log maps to the generic fiber $\cL_t \cong Z$ will not intersect the central fiber, and the log structure of $\cL$ restricted to $\cL_t$ is trivial. After forgetting the log structure, the stable log moduli space to $\cL_t$ is isomorphic to the ordinary moduli space of stable maps to $Z$, and the log Gromov-Witten invariants associated to $\barM(\cL_t(\log \cL_E), \beta+h)$ are equal to the ordinary Gromov-Witten invariants associated to $\barM(Z, \beta+h)$. We take the divisorial log structure on $\mathbb{A}^1$ with respect to $\{0\}$. As $\cL\rightarrow \mathbb{A}^1$ is a normal crossings degeneration, it is log smooth. By \cite{GS13}, the moduli space $\barM(\cL/\mathbb{A}^1, \beta+h)$ is proper.

We have the following lemma adapted from Lemma 2.2, \cite{vGGR} relating $[\barM(\cL_t)]^{vir}$ to $[\barM(\cL_0)]^{vir}.$

\begin{lemma}
\label{lem:deg1}
Let $P_0:\barM(\mathcal{L}_0(\text{log }\cL_E), \beta+h) \rightarrow \barM(X, \beta)$ be the map that forgets the log structure, composes with the natural maps $\cL_0 \rightarrow \mathscr{X}_0 \rightarrow X$, and stabilizes, and $P_t:\barM(Z, \beta+h) \rightarrow \barM(X, \beta)$ be the map that composes with the projection $Z \rightarrow X$, and stabilizes. Let $P:\barM(\mathcal{L}(\text{log }\mathcal{L}_0), \beta+h) \rightarrow  \barM(X\times\mathbb{A}^1/\mathbb{A}^1, \beta)$ be the map of moduli spaces that restricts to $P_t$ for $t \neq 0$ or $P_0$. Let $\barM(X\times \mathbb{A}^1/\mathbb{A}^1, \beta)$ be the space of ordinary stable maps to $X\times \mathbb{A}^1$ in curve class $\beta.$ When $t \neq 0$, we have the following equality of virtual cycles,

\[
(P_0)_* [\barM(\mathcal{L}_0(\log \cL_E), \beta+h)]^{vir} = (P_t)_* [\barM(Z, \beta+h)]^{vir} 
\]
\begin{proof}
We have the following commutative diagram,

\begin{center}
\begin{tikzcd}
\barM(\mathcal{L}_0(\log \cL_E), \beta+h) \arrow[hookrightarrow]{r} \arrow[d, "P_0"] & \barM(\mathcal{L}(\text{log }\mathcal{L}_0), \beta+h) \arrow[d, "P"] & \arrow[hook']{l} \barM(Z, \beta+h) \arrow[d, "P_t"]\\
\barM(X, \beta) \arrow[r] \arrow[d] &  \barM(X \times \mathbb{A}^1, \beta) \arrow[d, "f"] & \arrow[l] \barM(X, \beta) \arrow[d] \\
\{0\} \arrow[hookrightarrow, "i_0"]{r} & \mathbb{A}^1 & \arrow[hook', "i_t"']{l} \{t\}
\end{tikzcd}
\end{center}

We have the following equalities, 
\begin{align*}
    (P_0)_*[\barM(\cL_0)]^{vir} &= (P_0)_* i_0^{!}[\barM(\cL/\mathbb{A}^1]^{vir} \\
    &= i_0^{!} P_*[\barM(\cL/\mathbb{A}^1]^{vir} \\
    &= i_t^{!} P_*[\barM(\cL/\mathbb{A}^1]^{vir} \\
    &= (P_t)_* i_t^{!}[\barM(\cL/\mathbb{A}^1]^{vir} \\
    &= (P_t)_* [\barM(Z, \beta+h)]^{vir}
\end{align*}
The 1st and 5th equalities follow from $[\barM(\cL_0)]^{vir} = i_0^{!}[\barM(\cL/\mathbb{A}^1)]^{vir}$ and $[\barM(Z)]^{vir} = i_t^{!}[\barM(\cL/\mathbb{A}^1)]^{vir}$, because of compatibility of virtual classes with base change. The 2nd and 4th equalities follow from commutativity of Gysin pullback with proper pushforward applied to the top left and right squares. The 3rd equality follows because $f$ is the trivial family. 
\end{proof}
\end{lemma}

\subsection{Degeneration formalism}

\label{sec:degeneration_formula}

We briefly recall the set-up to apply the degeneration formula \cite{KLR}. Let $\barM(\cL_0) := \barM_{g,n+r}(\cL_0(\log \cL_E), \beta+h)$ be the moduli space of genus-$g$, basic stable log maps  with $n$ interior marked points and $r$ relative marked points to $\cL_0(\log \cL_E)$ of curve class $\beta+h$ (see Remark \ref{rem:deg_curve_class}). Stable log maps to the singular fiber $\cL_0$ are represented by bipartite graphs $\Gamma.$ Denote the vertices and edges of $\Gamma$ as $V(\Gamma)$ and $E(\Gamma)$ respectively. We assign to each vertex $V \in V(\Gamma)$ a genus $g_V \geq 0$, a curve class $\beta_V \in H_2(\cL_0)$, and a subset of markings $n_V \subset \{1,2,\ldots, n\}$. Each edge $e \in E(\Gamma)$ is assigned a non-negative integer weight $w_e \geq 0$, which represent relative contact orders with $\cL_E$. We have the following conditions satisfied by $\Gamma$ (see Section 2, \cite{KLR}),

\begin{align*}
    i_*\beta_X + p_*\beta_y &= \beta+h \\
    1 - \chi_{top}(\Gamma) + \sum_V g_V &= g \\
    \bigcup_V n_V &= \{1,2,\ldots, n\} \\
    \sum_e w_e &= \beta\cdot D
\end{align*}
Denote $\Gamma(g,n,\beta)$ to be the set of all bipartite graphs $\Gamma$ satisfying the above conditions.

For each vertex $V$, let $r_V$ be the number of its half-edges. Define the index $i(V)$ to be $X$ or $Y$ depending on if the target of $\barM_V$ is $\cL_X$ or $\cL_Y$, respectively. Let $\barM_V := \barM_{g_V, n_V + r_V}(\cL_{i(V)}(\log \cL_E), \beta_V)$ be the moduli space of genus-$g_V$, basic stable log maps to $\cL_{i(V)}(\log \cL_E)$ with $n_V$ interior marked points and $r_V$ relative marked points in curve class $\beta_V$. Let $ev_{\cL_{i(V)}}:\barM_V \rightarrow \cL_{i(V)}^{n_V}$ and $ev_{\cL_E}:\barM_V \rightarrow \cL_{E}^{r_V}$ be the evaluation map of interior and relative marked points, respectively. We say that a vertex $V$ is an \textit{$X$-vertex} or \textit{$Y$-vertex} if maps in $\barM_V$ map to $\cL_X$ or $\cL_Y$, respectively.

We have the following commutative diagram (see Section 1.3, \cite{KLR}),

\begin{center}
\begin{tikzcd}
 \odot_V \barM_V \arrow[d] \arrow[r] & \prod_V \barM_V \arrow[d, "ev"] \\
 \prod_e E \arrow[r, "\Delta"]{l} & \prod_V \prod_{V \in e}E \times E
\end{tikzcd}
\end{center}
which defines the space $\odot_V \barM_V$ as the fiber product of the diagonal map $\Delta$ and $ev$. A stable map in $\odot_V \barM_V$ satisfies the condition that if two vertices $V_1$ and $V_2$ are joined by an edge $e$, then maps in $\barM_{V_1}$  and $\barM_{V_2}$ will intersect at the same point in the divisor $\cL_E$ with the same contact order $w_e$; this is also known as the predeformability condition (see Section 2.2, \cite{GV}).

Let $\barM_{\Gamma}$ be the moduli space of stable maps whose dual intersection graph collapses to $\Gamma$ with a subset of its nodes corresponding to edges $e_1,\ldots, e_r.$ We have an \'etale map that partially forgets the log structure $\Phi: \barM_{\Gamma} \rightarrow \odot_V \barM_V$ with $\deg \Phi = \frac{\prod_e w_e}{lcm\{w_e\}}$ (Equation 1.4 of \cite{KLR}), and a finite map $F: \barM_{\Gamma} \rightarrow \barM(\cL_0)$ that forgets the graph marking of the stable map. 

\begin{theorem}[Theorem 1.5, \cite{KLR}]
    We have the equality of virtual classes,
    
    \[
    [\barM(\cL_0)]^{vir} = \sum_{\text{bipartite }\Gamma \in \Gamma(g,n,\beta)} \frac{lcm\{w_e\}}{|Aut(\Gamma)|} F_* \Phi^* \Delta^{!} \prod_V [\barM_V]^{vir}
    \]
\end{theorem}

Thus, the degeneration formula expresses $[\barM(\cL_0)]^{vir}$ as a sum over bipartite graphs $\Gamma \in \Gamma(g,n,\beta)$ of virtual classes $[\barM_{V}]^{vir}$ for $V \in V(\Gamma).$ 

\subsection{Bipartite graphs $\Gamma$}

In this section, let $B$ denote the $\bP^1$-fiber class of $\pi_Y: Y\rightarrow E$, and $F$ denote the $\bP^1$-fiber class of $\pi_{\cL_Y}: \cL_Y \rightarrow Y$. In genus-0, the proof for which bipartite graphs $\Gamma$ have a non-trivial contribution to the virtual class is done in \cite{Wan}. We explain how certain lemmas needed from \cite{vGGR} generalize to higher genus when $\dim X = 2.$

\begin{theorem}
\label{thm:deg-graphs}
Let $g \geq 0$. Then, $[\barM_{\Gamma}]^{vir} = 0$, unless $\Gamma$ is of the following form: the edge connecting vertices $V_1$ and $V_3$ has weight $w_{e_1} = 1$, and the edge connecting $V_1$ and $V_2$ has weight $w_{e_2} = \beta\cdot E - 1$. The curve class $\beta \in NE(X)$ is attached to vertex $V_1$, the class $(\beta\cdot E - 1)B$ is attached to vertex $V_2$, and the class $B + F$ is attached to vertex $V_3$. We have $g = g_{V_1} + g_{V_2} + g_{V_3}$ (see Figure \ref{fig:deg_graph}).

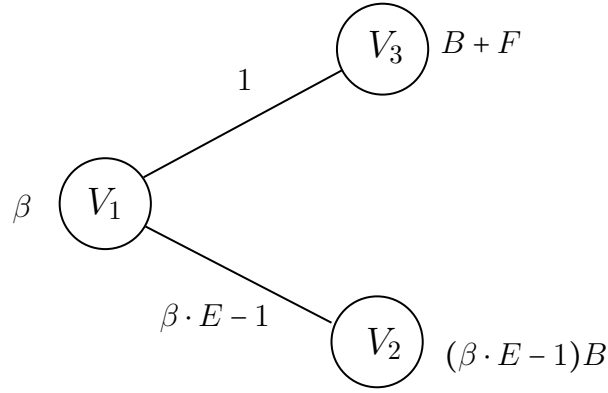
\begin{figure}[h!]
    \centering
    \tikzset{every picture/.style={line width=0.75pt}} 

\begin{tikzpicture}[x=0.75pt,y=0.75pt,yscale=-1,xscale=1]

\draw   (160.04,159.86) .. controls (160.04,147.26) and (170.29,137.05) .. (182.93,137.05) .. controls (195.57,137.05) and (205.82,147.26) .. (205.82,159.86) .. controls (205.82,172.45) and (195.57,182.66) .. (182.93,182.66) .. controls (170.29,182.66) and (160.04,172.45) .. (160.04,159.86) -- cycle ;
\draw   (296.47,229.19) .. controls (296.47,216.6) and (306.72,206.38) .. (319.36,206.38) .. controls (332,206.38) and (342.25,216.6) .. (342.25,229.19) .. controls (342.25,241.79) and (332,252) .. (319.36,252) .. controls (306.72,252) and (296.47,241.79) .. (296.47,229.19) -- cycle ;
\draw   (299.21,82.31) .. controls (299.21,69.71) and (309.46,59.5) .. (322.11,59.5) .. controls (334.75,59.5) and (345,69.71) .. (345,82.31) .. controls (345,94.9) and (334.75,105.12) .. (322.11,105.12) .. controls (309.46,105.12) and (299.21,94.9) .. (299.21,82.31) -- cycle ;
\draw    (202.16,146.63) -- (301.96,92.8) ;
\draw    (203.07,171.26) -- (296.47,219.61) ;

\draw (172.14,150.26) node [anchor=north west][inner sep=0.75pt]  [font=\large]  {$V_1$};
\draw (313.19,71.8) node [anchor=north west][inner sep=0.75pt]  [font=\large]  {$V_3$};
\draw (312.36,220.5) node [anchor=north west][inner sep=0.75pt]  [font=\large]  {$V_2$};
\draw (247.43,91.13) node [anchor=north west][inner sep=0.75pt]    {$1$};
\draw (134.81,154.08) node [anchor=north west][inner sep=0.75pt]    {$\beta $};
\draw (351.61,227.07) node [anchor=north west][inner sep=0.75pt]    {$( \beta \cdot E-1) B$};
\draw (349.72,71.97) node [anchor=north west][inner sep=0.75pt]    {$B+F$};
\draw (208.74,208.82) node [anchor=north west][inner sep=0.75pt]    {$\beta \cdot E-1$};

\end{tikzpicture}
    \caption{Bipartite graphs in the degeneration}
    \label{fig:deg_graph}
\end{figure}

\end{theorem}
\subsubsection{Condition on the X-vertices}

In this section, we show that any $X$-vertex $V$ has at most 2 edges (see Section \ref{sec:degeneration_formula} for definition of $X$-vertex).

\begin{lemma}
\label{lem:deg2}
    Let $\Gamma$ be a graph with an $X$-vertex $V$ with $r_V > 2$ adjacent edges and $g_V \geq 0$, then $[\barM_{\Gamma}]^{vir} = 0.$ 
\begin{proof}
    Since non-surjective maps from a proper, genus-$g_V$ nodal curve to $\bP^1$ are constant, the evaluation map $\barM_V \rightarrow (E\times \bP^1)^{r_V}$ factors through $E^{r_V} \times \bP^1$, where $\bP^1$ is embedded diagonally. In addition, we separate out $\barM_V$ from $\prod_{V' \neq V} \barM_{V'}$. We have the following commutative diagram. 
    
    \begin{center}
    \begin{tikzcd}
    \barM_V \times_{\cL_E^r} \odot_{V' \neq V} \barM_{V'}\arrow[rightarrow]{r} \arrow[d, "ev"] & \barM_V \times \prod_{V' \neq V}\barM_{V'} \arrow[d, "ev"] \\
    (E^r \times \bP^1) \times (E \times \bP^1)^s \arrow[r, "\Delta'"] \arrow[d, "\delta := (id \times diag) \times id"] &  (E^r \times \bP^1)^2 \times (E \times \bP^1)^{2s}  \arrow[d, "\delta' := (id \times diag) \times id"]  \\
    (E \times \bP^1)^r \times (E \times \bP^1)^s \arrow[rightarrow, "\Delta"]{r} & (E \times \bP^1)^{2r} \times (E \times \bP^1)^{2s} 
    \end{tikzcd}
    \end{center}
    
    Let $N, N'$ be the normal bundles of $\Delta, \Delta'$, respectively. Define $A := \delta^* N/N'$, with rank $r-1.$ The excess intersection formula (\cite{Ful}, Theorem 6.3) tells us that,
    
    \[
    \Delta^{!} \alpha = c_{r-1}(ev^* A) \cap (\Delta')^{!} \alpha
    \]
    for $\alpha \in A_*(\barM_V \times \prod_{V' \neq V}\barM_{V'}).$
    
    The normal bundle $M$ of $\delta$ is $T(\bP^1)^{r-1} = \cO_{\bP^1}(2)^{r-1}$, and the normal bundle $M'$ of $\delta'$ is isomorphic to $\cO_{\bP^1}(2)^{2r-2}.$  By the Cartesian property of the bottom square, we have that $A \cong (\Delta')^* M'/M \cong \cO_{\bP^1}(2)^{r-1}.$ We see that $c_{r-1}(A) = 0$ if $r > 2.$ Applying this to the class $[\barM_V]^{vir} \times \prod_{V' \neq V} [\barM_{V'}]^{vir}$, we have the desired result.
\end{proof}
\end{lemma}

\subsubsection{Conditions on the Y-vertices}

Let $V$ be a $Y$-vertex. Recall that we have the projections $\pi_{\cL_Y}: \cL_Y \rightarrow Y$ and $\pi_Y: Y \rightarrow E.$  We first show that $(\pi_{\cL_Y})_*\beta_V$ must be a multiple of a fiber class of $Y \rightarrow E$. 

\begin{lemma}
\label{lem:deg3}
Let $g_V \geq 0$. If the curve class $(\pi_{\cL_Y})_*\beta_V$ is not a multiple of a fiber class of $Y \rightarrow E$, then $(ev_{\cL_E})_*[\barM_V]^{vir} = 0$.

\begin{proof}
This lemma is a higher genus version of Proposition 5.3, \cite{vGGR}. We write $Y(\log E_0)$ to be the log scheme $Y$ given by the divisorial log structure $E_0$. We have the following commutative diagram (Diagram 4.1, \cite{vGGR}), 

\begin{equation}
\label{eq:costello_diagram}
     \begin{tikzcd}
    \barM(Y(\log E_0), (\pi_{\cL_Y})_*\beta_V) \arrow[r, "u"] \arrow[d] & \cM \arrow[r, "v"]\arrow[d] & \barM(E, (\pi_Y\circ  \pi_{\cL_Y})_*\beta_V) \arrow[d] \\
    \mathcal{M}_{g,n,H_2(Y)^{+}}^{log} \arrow[r, "id"] & \mathcal{M}_{g, n, H_2(Y)^{+}}^{log} \arrow[r, "\nu"] & \mathcal{M}_{g,n, H_2(E)^{+}}
    \end{tikzcd}   
\end{equation}
where $\mathcal{M}_{g,n,H_2(Y)^{+}}^{log}$ is the stack of genus-$g$, $n$-marked, pre-stable log curves that additionally remembers the curve class of each irreducible component \cite{Cos}. The space $\cM$ is defined to make the right hand square Cartesian, and its obstruction theory is the pullback obstruction theory by $\nu$. By \cite{Man}, we have,

\[
\nu^{!}[\barM_{g,n}(E, (\pi_Y \circ \pi_{\cL_Y})_* \beta_V)]^{vir} = [\cM]^{vir}
\]
We also have the following short exact sequence,

\[
0 \rightarrow T(Y(\log E_0))^{log}/E_0 \rightarrow T(Y(\log E_0))^{log} \rightarrow T E_0 \rightarrow 0
\]
where $T(Y(\log E_0))^{log}$ is the log tangent bundle of $Y(\log E_0)$, and $T(Y(\log E_0))^{log}/E_0$ is the log tangent bundle of $Y(\log E_0)$ relative to $E_0$. The sequence induces a compatible triple for the left hand square in Diagram \ref{eq:costello_diagram}, and we have a well-defined virtual pullback $u^{!}$ \cite{Man}. Diagram \ref{eq:costello_diagram} is used to prove Theorem 4.1, \cite{vGGR}, which we state for convenience here,
\begin{theorem}[Theorem 4.1, \cite{vGGR}]
    \label{thm:vGGR1}
    Let $\pi_Y:Y \rightarrow E$ be a log smooth morphism where $E$ has trivial log structure. Suppose that for every log smooth morphism $f:C \rightarrow Y$ of genus $g$ and class $\beta$ we have $H^1(C,  f^* TY^{log}) = 0$, then
    
    \[   
    \barM_{g,n}(Y(\log E_0), (\pi_{\cL_Y})_*\beta_V)  = u^* \nu^{!}[\barM_{g,n}(E, (\pi_Y\circ  \pi_{\cL_Y})_* \beta)]^{vir}
    \]
    provided that $[\barM_{g,n}(E, (\pi_Y\circ  \pi_{\cL_Y})_* \beta_V)]^{vir} \neq 0.$ Hence, if $[\barM_{g,n}(E, (\pi_Y\circ  \pi_{\cL_Y})_* \beta_V)]^{vir}$ is represented by a cycle supported on the locus $W \subset \barM_{g,n}(E, (\pi_Y\circ  \pi_{\cL_Y})_* \beta_V)$, then $[\barM_{g,n}(Y(\log E_0), \beta_V)]^{vir}$ is represented by a cycle supported on $w^{-1}(W)$ where $w = v \circ u.$ 
\end{theorem}
The convexity assumption in Theorem \ref{thm:vGGR1} guarantees that $u$ is smooth, in which case $u^{!}$ agrees with smooth pullback $u^*$. For our purposes, we relax the convexity assumption in Theorem \ref{thm:vGGR1}, and it suffices to extend Theorem 4.1 to higher genus in the non-convex setting. We use $u^{!}$ instead of $u^*.$ By \cite{Man}, Corollary 4.9, the virtual classes are related by, 

\[
    [\barM_{g,n}(Y(\log E_0), (\pi_{\cL_Y})_*\beta_V)]^{vir}  = u^{!} \nu^{!}[\barM_{g,n}(E, (\pi_Y\circ  \pi_{\cL_Y})_* \beta)]^{vir}
\]
In particular, if $[\barM(E, \pi_{Y*}\beta)]^{vir}$ is a cycle supported on some locus $W$, then $[\barM(Z, \beta)]^{vir}$ is a cycle supported on $v^{-1}u^{-1}(W)$.
Consequently, Lemma 5.1 and Proposition 5.3 of \cite{vGGR} imply that if $(\pi_{\cL_Y})_* \beta_V$ is not a multiple of a fiber class of $Y \rightarrow E$, then $(ev_{\cL_E})_*[\barM_V]^{vir} = 0$. 

\end{proof}
\end{lemma}

Recall the definitions of $r_V, n_V$ from Section \ref{sec:degeneration_formula}. By Lemma \ref{lem:deg3}, the curve classes of $Y$-vertices are multiples $kB$ of the $\bP^1$-fiber $B$ of $Y \rightarrow E$, or $kB + F$ where $F$ is a $\bP^1$-fiber of $\cL_Y \rightarrow Y$.  Therefore, the relative evaluation map $ev_{\cL_E}: \barM(\cL_Y(\log \cL_E), \beta_V) \rightarrow (E\times \bP^1)^{r_V}$ factors as,

\begin{align*}
 &ev_{\cL_E}: \barM(\cL_Y(\log \cL_E), \beta_V) \xrightarrow{\pi_{\cL_Y}} \barM(Y(\log E_0), (\pi_{\cL_Y})_*\beta_V) \xrightarrow{p} \barM(E, (\pi_Y \circ \pi_{\cL_Y})_*\beta_V) \\ &\xrightarrow{ev} E \times 0 \hookrightarrow (E\times \bP^1)^{r_V}   
\end{align*}
The map $p$ forgets the log structure, composes with $\pi_Y$, and stabilizes. 

Next, we show that a $Y$-vertex $V$ has at most a single edge, i.e. $r_V \leq 1$.

\begin{lemma}
\label{lem:deg4}

    Let $V$ be a $Y$-vertex with $g_V \geq 0$. Suppose that $r_V > 1$. Denote $[pt_{\cL_Y}] \in A^3(\cL_Y)$ to be the Poincar\'e dual of a point in $\cL_Y$. If the curve class of $V$ is 1) $kB$, then $(ev_{\cL_E})_*([\barM_V]^{vir}) = 0$, or 2) $kB + F$, then $(ev_{\cL_E})_*([pt_{\cL_Y}] \cap [\barM_V]^{vir}) = 0$, 

    \begin{proof}
        In case 1), suppose the curve class is $kB$. The integer $k$ is also the weight of the edge connected to $V$. By Lemma 5.4, \cite{vGGR}, we have $n_V = 0$. The virtual dimension of $\barM_V = \barM_{g_V, r_V}(\cL_Y(\log \cL_D), kB)$ is $r_V$. Since the evaluation map $ev_{\cL_E}$ factors through $E \times 0 \hookrightarrow E \times \bP^1$, then if $r_V > 1$, we have $(ev_{\cL_E})_*([\barM_V]^{vir}) = 0$ by degree vanishing.

        In case 2), the vertex $V$ must contain the interior marked point, i.e. $n_V = 1.$ The virtual dimension of $\barM_V = \barM_{g_V, 1+r_V}(\cL_Y(\log \cL_D), kB+F)$ is $3+r_V.$ The evaluation map factors through $E \times 0 \hookrightarrow E \times \bP^1$. Hence, if $r_V > 1$, we have $(ev_{\cL_E})_*([pt_{\cL_Y}] \cap [\barM_V]^{vir}) = 0$ by degree vanishing.
    \end{proof}

\end{lemma}

\begin{remark}
    Since $\dim X = 2$, the virtual dimension of $\barM(\cL_Y)$ does not depend on genus $g_V$. Thus, the degree vanishing of virtual classes in Lemma \ref{lem:deg4} holds for all $g_V \geq 0$ .
\end{remark}

\begin{proof}[Proof of Theorem \ref{thm:deg-graphs}]
    By Lemmas \ref{lem:deg2}, \ref{lem:deg3}, \ref{lem:deg4}, bipartite graphs are of the form in Figure \ref{fig:deg_graph} for all genus $g \geq 0$. Now, suppose by contradiction that $w_{e_1} \geq 2$. Since $B \cdot (w_{e_1}B + F) < 0$, the curve class $\beta_V$ contains a fiber $B$ of $Y \rightarrow E$, and is reducible. Then, the evaluation map $ev_{\cL_E}$ of $\barM_V$ factors through a point $pt_E \hookrightarrow E\times 0 \hookrightarrow E \times \bP^1$. \footnote{Explicitly, if $[pt_{\cL_Y}] \in A_0(\cL_Y)$ represents the fixed point that curves $B+F$ pass through, then $pt_E = (\pi_Y)_* (\pi_{\cL_Y})_*([pt_{\cL_Y}]) \in A_0(E)$.} Hence, $(ev_{\cL_E})_*[\barM_V]^{vir} = 0$ for all $g \geq 0$ by degree vanishing. Thus, the edge weights of $\Gamma$ are $w_{e_1} = 1$ and $w_{e_2} = \beta\cdot E - 1$: 
\end{proof}

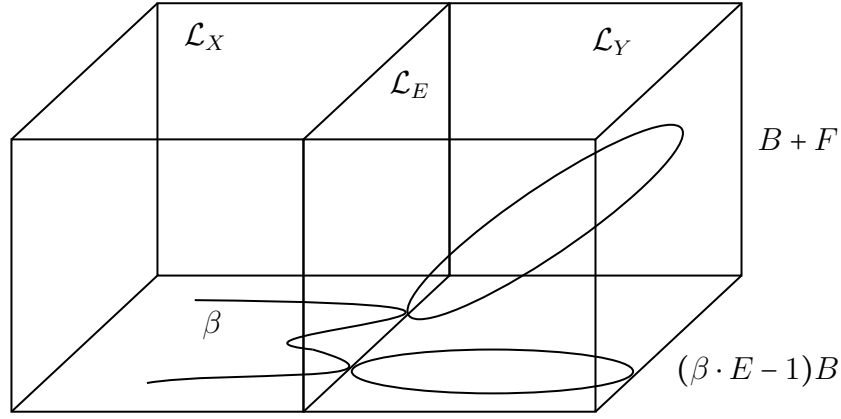
\begin{figure}[h!]
\centering

\tikzset{every picture/.style={line width=0.75pt}} 

\begin{tikzpicture}[x=0.90pt,y=0.90pt,yscale=-1,xscale=1]

\draw   (143.67,217.44) -- (265.8,217.44) -- (265.8,331.33) -- (143.67,331.33) -- cycle ;
\draw   (204.73,160.5) -- (326.87,160.5) -- (326.87,274.39) -- (204.73,274.39) -- cycle ;
\draw   (265.8,217.44) -- (387.93,217.44) -- (387.93,331.33) -- (265.8,331.33) -- cycle ;
\draw   (326.87,160.5) -- (449,160.5) -- (449,274.39) -- (326.87,274.39) -- cycle ;
\draw    (143.67,217.44) -- (204.73,160.5) ;
\draw    (265.8,217.44) -- (326.87,160.5) ;
\draw    (143.67,331.33) -- (204.73,274.39) ;
\draw    (265.8,331.33) -- (326.87,274.39) ;
\draw    (387.93,217.44) -- (449,160.5) ;
\draw    (387.93,331.33) -- (449,274.39) ;
\draw    (200.33,319.33) .. controls (231.67,314) and (320.33,320) .. (269.67,305.33) ;
\draw    (220.33,284.67) .. controls (419,286.67) and (211.67,299.33) .. (269.67,305.33) ;
\draw   (286,314.5) .. controls (286,309.44) and (312.34,305.33) .. (344.83,305.33) .. controls (377.33,305.33) and (403.67,309.44) .. (403.67,314.5) .. controls (403.67,319.56) and (377.33,323.67) .. (344.83,323.67) .. controls (312.34,323.67) and (286,319.56) .. (286,314.5) -- cycle ;
\draw   (343.39,252) .. controls (372.43,229.54) and (406.5,211.33) .. (419.48,211.33) .. controls (432.47,211.33) and (419.46,229.54) .. (390.42,252) .. controls (361.38,274.46) and (327.31,292.67) .. (314.33,292.67) .. controls (301.34,292.67) and (314.35,274.46) .. (343.39,252) -- cycle ;

\draw (214.67,167.73) node [anchor=north west][inner sep=0.75pt]    {$\mathcal{L}_{X}$};
\draw (300.67,187.73) node [anchor=north west][inner sep=0.75pt]    {$\mathcal{L}_{E}$};
\draw (385.33,169.73) node [anchor=north west][inner sep=0.75pt]    {$\mathcal{L}_{Y}$};
\draw (420.47,306.26) node [anchor=north west][inner sep=0.75pt]    {$( \beta \cdot E-1) B$};
\draw (454.47,211.59) node [anchor=north west][inner sep=0.75pt]    {$B+F$};
\draw (222.33,288.07) node [anchor=north west][inner sep=0.75pt]    {$\beta $};

\end{tikzpicture}

    \caption{Stable log map to the central fiber $\cL_0 = \cL_X \sqcup_{\cL_E} \cL_Y$ represented by a bipartite graph in Theorem \ref{thm:deg-graphs}.} 
    \label{fig:curve-splitting}
\end{figure}

\subsection{Invariants associated to $\Gamma$}

Let $\beta \in H_2(X,\mathbb{Z})$ be a curve class, and define $e := \beta\cdot E$. By Theorem \ref{thm:deg-graphs}, the degeneration formula gives the following equality of virtual classes,

\begin{equation}
\label{eq:splitting_virtual}
\begin{split}
  [\barM(\cL_0)]^{vir} &= \sum_{\substack{\Gamma \in \Gamma(g,n,\beta), \\g_A + g_B + g_C = g}} (e-1)F_* \Phi^* \Delta^{!}\Bigl([\barM_{g_A, 2}(\cL_X(\log \cL_E), \beta)]^{vir} \times \\ &[\barM_{g_B, 1}(\cL_Y(\log \cL_E), (e-1)B)]^{vir} \times [\barM_{g_C, 2}(\cL_Y(\log \cL_E), B+F)]^{vir}\Bigr) 
\end{split}
\end{equation}
(see Section \ref{sec:degeneration_formula} for the definition of $F, \Phi,$ and $\Delta$), and Figure \ref{fig:curve-splitting} for curves represented by $\Gamma$.

In forming log invariants from the log smooth family $\cL \rightarrow \mathbb{A}^1$, we cap $[\barM(\cL)]^{vir}$ with an incidence condition $\alpha \in A_1(\cL)$, defined as $\alpha_t := [pt] \in H^6(Z, \mathbb{Z})$ for $t \neq 0$, and $\alpha_0 := [pt] \in H^6(\cL_Y, \mathbb{Z})$. If $i_t: \cL_t \hookrightarrow \cL$ is the inclusion for $t \neq 0$, then,

\[
i_t^{!} \alpha \cap [\barM(\cL)]^{vir} = N_{g,1}(Z, \beta+h)
\]
By log smooth invariance of logarithmic Gromov-Witten theory (Appendix A, \cite{MR}), we have,

\[
N_{g,1}(Z, \beta+h) = \alpha_0 \cap [\barM(\cL_0)]^{vir}
\]
We use the degeneration formula to calculate $\alpha_0 \cap [\barM(\cL_0)]^{vir}$ by applying $\alpha_0$ to the right side of Equation \ref{eq:splitting_virtual}.

\subsubsection{Invariants of Vertex $V_1$}

\label{sec:VertexA}

For vertex $V_1$ in Theorem \ref{thm:deg-graphs}, stable log maps have 0 interior marked points and 2 relative marked points, i.e. $n_{V_1} = 0$ and $r_{V_1} = 2$. To vertex $V_1$, we associate the moduli space $\barM_{V_1} := \barM_{g_{V_1},2}(\cL_X(\log \cL_E), \beta)$ of genus-$g_{V_1}$, basic stable log maps to $\cL_X(\log \cL_E)$ with 2 relative contact points in curve class $\beta$. The virtual dimension of $\barM_{V_1}$ is 2. 

Define $\gamma_{\cL_E} := [pt_E \times \bP^1] \in A_1(\cL_E)$, where $pt_E \hookrightarrow E$ is a point. This class will serve as a relative contact condition in the degeneration formula for maps in $\barM_{V_1}$. Since $[pt_E \times \bP^1]$ intersected with $[E \times pt_{\bP^1}] = 1$, the Poincar\'e dual of $\gamma_{\cL_E}$ is the class $\gamma_{\cL_E}^{\vee} := [E \times pt_{\bP^1}] \in A_1(\cL_E)$. Later on, the class $\gamma_{\cL_E}^{\vee}$ will serve as a relative contact condition for stable log maps in $\barM_{V_3}$.

We have the following Cartesian diagram by restricting the evaluation map to $E \times 0$, 

\begin{center}
  \begin{tikzcd}
   \barM_{g_{V_1},2}(X(\log E), \beta)_{(1, e-1)} \arrow[r] \arrow[d, "ev"] & \barM_{V_1} \arrow[d, "ev"] \\
    E \times 0 \arrow[r, hook, "i"] & E \times \bP^1
\end{tikzcd}    
\end{center}
where the fiber product $\barM_{g_{V_1},2}(X(\log E), \beta)_{(1, e-1)}$ denotes the moduli space of genus-$g_{V_1}$, basic stable log maps to $X(\log E)$ in curve class $\beta$, with one prescribed contact point to $E$ of order 1, and a non-prescribed contact point to $E$ of order $e-1.$ We have $\vdim \barM_{g_{V_1},2}(X(\log E), \beta)_{(1, e-1)} = g_{V_1} + 1.$ 

Let $\mathcal{U} \xrightarrow{ft} \barM_{g_{V_1},2}(X(\log E), \beta)_{(1, e-1)}$ be the universal curve of $\barM_{g_{V_1},2}(X(\log E), \beta)_{(1, e-1)}$ with evaluation map $\mathcal{U} \xrightarrow{ev} X$. Since the normal bundle $N_{X \times 0/ X \times \bP^1}$ is isomorphic to $\cO_X$, the virtual classes of $\barM_{g_{V_1},2}(X(\log E), \beta)_{(1, e-1)}$ and $\barM_{V_1}$ differ by $e(R^1 ft_* ev^* N_{X \times 0/ X \times \bP^1}) = (-1)^g \lambda_g$. Hence, we have,

\[
i^{!}[\barM_{V_1}]^{vir} = (-1)^{g_{V_1}} \lambda_{g_{V_1}} \cap [\barM_{g_{V_1},2}(X(\log E), \beta)_{(1, e-1)}]^{vir}
\]
To vertex $V_1$, we associate the class,

\[
\gamma_{\cL_E} \cap [\barM_{V_1}]^{vir}
\]
We calculate its Gysin pullback,

\[
i^{!}(\gamma_{\cL_E} \cap [\barM_{V_1}]^{vir}) = i^*\gamma_{\cL_E} \cap i^{!}[\barM_{V_1}]^{vir} = [pt_E] \cap (-1)^{g_{V_1}} \lambda_{g_{V_1}} \cap [\barM_{g_{V_1},(1, e-1)}(X(\log E), \beta)]^{vir}
\]
since $i^* \gamma_{\cL_E} = \gamma_{\cL_E} \cdot (E \times 0) = [pt_E] \in A_0(E) = A^1(E).$ 

The invariant associated to vertex $V_1$ is,

\begin{equation}
\label{eq:vertexA_invariant}
    R_{g_{V_1}, (1,e-1)}(X(\log E), \beta) := \int_{[\barM_{g_{V_1},2}(X(\log E), \beta)_{(1, e-1)}]^{vir}}(-1)^{g_{V_1}}\lambda_{g_{V_1}}ev^*([pt_E])
\end{equation}
where $[pt_E] \in A^1(E)$. This is the same invariant defined in Section 4, \cite{GRZ}, and can be computed via $q$-refined tropical curve counting \cite{Gra}.

\subsubsection{Invariants of Vertex $V_2$} 
\label{sec:VertexB}

For vertex $V_2$, by Lemma 5.4 of \cite{vGGR}, stable log maps have 0 interior marked points and 1 relative marked point, i.e. $n_{V_2} = 0$ and $r_{V_2} = 1$. To vertex $V_2$, we associate the moduli space $\barM_{V_2} := \barM_{g_{V_2}}(\cL_Y(\log \cL_E), (e-1)B)$ of genus-$g_{V_2}$, basic stable log maps to $\cL_Y(\log \cL_D)$ in curve class $(e-1)B$ with 1 relative contact point of maximal tangency order $(e-1).$ The virtual dimension of $\barM_{V_2}$ is $1$. The evaluation map $ev_{\cL_E}$ factors through $E \times 0 \hookrightarrow E \times \bP^1$ since the curve class is a multiple of the fiber class $B$. Recall that $\cL_Y$ restricted to a fiber of $Y \rightarrow E$ is $\FF_1 = \bP(\cO_{\bP^1}(-1) \oplus \cO_{\bP^1})$, since $B \cdot -E_{\infty} = -1$ \footnote{
Recall that the intersection theory of $\FF_1 = \bP(\cO_{\bP^1}(-1) \oplus \cO_{\bP^1})$ is given by the ring $\mathbb{Z}[B,F]/\langle B^2 = -1, B \cdot F = 1, F \cdot F = 0\rangle$, where $B, F \in H_2(\FF_1, \mathbb{Z})$. Effective curve classes $nB + mF \in H_2(\FF_1, \mathbb{Z})$ satisfy $m, n \geq 0.$ The toric boundary of $\FF_1$ as a toric variety is the anti-canonical class $2B+3F$.}.

We have the following Cartesian diagram,

\begin{center}
  \begin{tikzcd}
   \barM_{g_{V_2}}(\FF_1(\log F), (e-1)B) \arrow[r] \arrow[d, "ev"] & \barM_{V_2} \arrow[d, "ev"] \\
    pt_E \arrow[r, hook, "i"] & E \times 0
\end{tikzcd}    
\end{center}
where $pt_E \in E \times 0$ is a point, and the fiber product $\barM_{g_{V_2}}(\FF_1(\log F), (e-1)B)$ denotes the moduli space of genus-$g_{V_2}$, basic stable log maps to $\FF_1(\log F)$ in curve class $(e-1)B$ with 1 relative contact point of maximal tangency order $(e-1).$ We have $\vdim \barM_{g_{V_2}}(\FF_1(\log F), (e-1)B) = g_{V_2}$. 

Let $\mathcal{U} \xrightarrow{ft} \barM(\FF_1(\log F), (e-1)B)$ be the universal curve of $\barM(\FF_1(\log F), (e-1)B)$, with evaluation map $\mathcal{U} \xrightarrow{ev} \FF_1$. The obstruction theories defining $[\barM_{V_2}]^{vir}$ and \\$[\barM(\FF_1(\log F), (e-1)B)]^{vir}$ differ by $e(R^1 ft_* ev^* N_{\FF_1/\cL_Y})$. Since $N_{\FF_1/\cL_Y} \cong \cO_{\FF_1}$, we have the relation,

\[
i^{!}[\barM_{V_2}]^{vir} = (-1)^{g_{V_2}} \lambda_{g_{V_2}} \cap [\barM_{g_{V_2},1}(\FF_1(\log F), (e-1)B)]^{vir}
\]
where $i^{!}$ denotes the Gysin pull back. The invariant associated to vertex $V_2$ is,

\begin{equation}
\label{eq:vertexB_invariant}
    R_{g_{V_2}}(\FF_1(\log F), (e-1)B) := \int_{[\barM_{g_{V_2}}(\FF_1(\log F), (e-1)B)]^{vir}}(-1)^{g_{V_2}}\lambda_{g_{V_2}}
\end{equation}
For virtual dimension reasons, there are no relative contact conditions with $\cL_E$ coming from the degeneration formula. Since $B \cong \bP^1$ has normal bundle $N_{B/\FF_1} \cong \cO_{\bP^1}(-1)$, any genus-$g_{V_2}$ stable log map in $\barM_{g_{V_2}}(\FF_1(\log F), (e-1)B)$ factors through $B$. Let $\barM_{g_{V_2}}(\bP^1(\log \infty), (e-1)[\bP^1])$ denote the moduli space of genus-$g_{V_2}$ basic stable log maps to $\bP^1(\log \infty)$ in curve class $(e-1)[\bP^1]$ with 1 relative contact point of maximal tangency. We have the identification of moduli spaces,

\[
\barM_{g_{V_2}}(\FF_1(\log F), (e-1)B) = \barM_{g_{V_2}}(\bP^1(\log \infty), (e-1)[\bP^1])
\]
Let $\mathcal{U} \xrightarrow{ft} \barM_{g_{V_2}}(\bP^1(\log \infty), (e-1)[\bP^1])$ be the universal curve of $\barM_{g_{V_2}}(\bP^1(\log \infty), (e-1)[\bP^1])$, with  evaluation map $\mathcal{U} \xrightarrow{ev} \bP^1 \cong B$. The obstruction theories defining \\$[\barM_{g_{V_2}}(\FF_1(\log F), (e-1)B)]^{vir}$ and $[\barM_{g_{V_2}}(\bP^1(\log \infty), (e-1)[\bP^1])]^{vir}$ differ by \\$e(R^1\pi_* f^* N_{B/\FF_1})$. The invariant in \ref{eq:vertexB_invariant} is equivalently,

\begin{equation}
\label{eq:vertexB_invariant_also}
\int_{[\barM_{g_{V_2}}(\bP^1(\log \infty), (e-1)[\bP^1]]^{vir})}e(R^1\pi_* f^* (\cO_{\bP^1} \oplus \cO_{\bP^1}(-1)))
\end{equation}
where are computed by Theorem 5.1, \cite{BP}; the genus-$g_{V_2}$ invariant is the coefficient of $\hbar^{2g_{V_2}}$ in the expression,

\[
\frac{(-1)^e}{(e-1)}\frac{i \hbar}{q^{(e-1)i\hbar/2} - q^{-(e-1)i\hbar/2}} = \frac{\hbar}{2}\csc \frac{(e-1)\hbar}{2}
\]
where $q = e^{i\hbar}$. The first few terms of $\frac{\hbar}{2}\csc \frac{(e-1)\hbar}{2}$ are $\frac{(-1)^e}{(e-1)^2} + \frac{(-1)^e}{24}\hbar^2 + \frac{7(-1)^e(e-1)^2}{5760}\hbar^4 + \ldots$.

\subsubsection{Invariants of Vertex $V_3$}

\label{sec:VertexC}

In Theorem \ref{thm:deg-graphs}, Vertex $V_3$ must contain the interior marked point and has 1 relative marked point, i.e. $n_{V_3} = r_{V_3}  = 1$. To Vertex $V_3$, we associate the moduli space $\barM_{V_3}  := \barM_{g_{V_3} ,2}(\cL_Y(\log \cL_E), B+F)$ of genus-$g_C$, basic stable log maps to $\cL_Y(\log \cL_E)$ in curve class $B+F$ with 1 fixed interior point and 1 relative contact point of order 1. We have $\vdim \barM_{V_3} = 4$. Recall that we have the evaluation maps $ev_{\cL_Y}:\barM_{V_3} \rightarrow \cL_Y$ and $ev_{\cL_E}: \barM_{V_3} \rightarrow \cL_E$. The evaluation map $ev_{\cL_E}$ factors through $E \times 0 \hookrightarrow E \times \bP^1$.

For Vertex $V_3$, we compute the invariant,

\begin{equation}
\label{eq:vertexC_locus}
  [\barM_{V_3}]^{vir} \cap ev_{\cL_Y}^*[pt_{\cL_Y}] \cap ev_{\cL_E}^*\gamma_{\cL_E}^{\vee}  
\end{equation}
where $\gamma_{\cL_E}^{\vee}$ was described in Section \ref{sec:VertexA}. The point constraint $[pt_{\cL_Y}]$ comes from the pullback of the incidence condition $i_{0}^{!}\alpha$ to the central fiber $\cL_0.$ The relative incidence condition $\gamma_{\cL_E}^{\vee} = [E \times pt_{\bP^1}] \in A_1(\cL_E)$ comes from the degeneration formula, as $[E \times pt_{\bP^1}] \cdot [pt_E \times \bP^1] = 1$.

We first have the following lemma showing that stable log maps in the class $B+F$ must have irreducible image,

\begin{lemma}
\label{lem:irreducible}
If the curve class $B+F$ of stable log maps in the locus $[\barM_C]^{vir} \cap [pt_{\cL_Y}] \cap [E \times pt_{\bP^1}]$ is reducible, then $(ev_{\cL_Y})_*([\barM_C]^{vir} \cap [pt_{\cL_Y}] \cap [E \times pt_{\bP^1}]) = 0.$
\begin{proof}
     If the class $B+F$ is reducible, then the evaluation map factors through a point $pt \hookrightarrow E \times 0 \hookrightarrow E \times \bP^1$, which is given by $pt = \pi_{Y*}\pi_{\cL_Y*} [pt_{\cL_Y}] \in A_0(E).$ Since $[\barM_C]^{vir} \cap ev_{\cL_Y}^*[pt_{\cL_Y}]$ is of degree 1, we have $(ev_{\cL_E})_*([\barM_C]^{vir} \cap ev_{\cL_Y}^*[pt_{\cL_Y}]) = 0$. Now, consider $(ev_{\cL_E})_*([\barM_C]^{vir} \cap ev_{\cL_Y}^*[pt_{\cL_Y}] \cap ev_{\cL_E}^*[E \times pt_{\bP^1}])$. By the projection formula, this is $(ev_{\cL_E})_*([\barM_C]^{vir} \cap ev_{\cL_Y}^*[pt_{\cL_Y}]) \cap (ev_{\cL_E})_* ev_{\cL_E}^*[E \times pt_{\bP^1}]$, which vanishes.
 \end{proof}
\end{lemma}

The evaluation map $ev_{\cL_E}$ factors through $E \times 0 \hookrightarrow E \times \bP^1$ (by the predeformability condition). We have the following Cartesian diagram,

\begin{center}
  \begin{tikzcd}
   \barM_{g_{V_3}, 2}(\FF_1(\log F), B+F) \arrow[r] \arrow[d, "ev"] & \barM_{V_3} \arrow[d, "ev"] \\
    pt_E \arrow[r, hook, "j"] & E \times 0
\end{tikzcd}    
\end{center}
where $pt_E \in E\times 0$ is a point, and the fiber product $\barM_{g_{V_3} ,2}(\FF_1(\log F), B+F)$ is the moduli space of genus-$g_{V_3}$, basic stable log maps to $\FF_1(\log F)$ in curve class $B+F$ with 1 fixed interior point and 1 relative contact point of order 1 \footnote{Taking the blow up at a point $\pi: \FF_1 \rightarrow \bP^2$, the class $B+F$ is the pullback $\pi^*H$ of the hyperplane class $H \in H_2(\bP^2, \mathbb{Z})$, and satisfies $(B+F)^2 = 1.$ The toric boundary of $\FF_1$ can also be written as $\partial\FF_1 = 2F + \pi^*H + B$.}. We have $\vdim \barM_{g_{V_3} ,2}(\FF_1(\log F), B+F) = g_{V_3} +3$. We have two evaluation maps \\$ev_{\FF_1}:\barM_{g_{V_3} ,2}(\FF_1(\log F), B+F) \rightarrow \FF_1$ and $ev_{F}:\barM_{g_{V_3} ,2}(\FF_1(\log F), B+F) \rightarrow F \cong \bP^1$. As $N_{\FF_1/ \cL_Y} \cong \cO_{\FF_1}$, the obstruction theories differ by $e(R^1 \pi_* f^* N_{\FF_1/ \cL_Y}) = (-1)^g \lambda_g$. The virtual classes hence are related by,

\[
j^{!}[\barM_{V_3}]^{vir} = (-1)^{g_{V_3}}\lambda_{g_{V_3}} \cap [\barM_{g_{V_3}, 2}(\FF_1(\log F), B+F)]^{vir}
\]
where $j^{!}$ is the Gysin-pullback. We evaluate the degree of \ref{eq:vertexC_locus} by the following local relative Gromov-Witten invariant of $\FF_1$ (see Figure \ref{fig:local_relative_F1}),  
\begin{equation}
    \label{eq:vertexC_invariant}
    R_{g_{V_3} , 2}(\FF_1(\log F), B+F) := \int_{[\overline{\mathcal{M}}_{g_{V_3} ,2}(\FF_1(\log F), B+F)]^{vir}} (-1)^{g_{V_3} }\lambda_{g_{V_3} }ev_{\FF_1}^*([pt_{\FF_1}]) ev_{\bP^1}^*([pt_{\bP^1}])
\end{equation}
where $[pt_{\FF_1}] \in A^{2}(\FF_1)$ is a fixed point in the interior of $\FF_1$ and $[pt_{\bP^1}] \in A^{1}(F)$ is a fixed point on $F \cong \bP^1$. In genus $g_{V_3} = 0$, $R_{g_{V_3} , 2}(\FF_1(\log F), B+F)$ is the number of lines through two points or 1. 

\begin{figure}
    \centering
    
\tikzset{every picture/.style={line width=0.75pt}} 

\begin{tikzpicture}[x=0.85pt,y=0.85pt,yscale=-1,xscale=1]

\draw   (285.8,237.44) -- (407.93,237.44) -- (407.93,351.33) -- (285.8,351.33) -- cycle ;
\draw   (346.87,180.5) -- (469,180.5) -- (469,294.39) -- (346.87,294.39) -- cycle ;
\draw    (285.8,237.44) -- (346.87,180.5) ;
\draw    (285.8,351.33) -- (346.87,294.39) ;
\draw    (407.93,237.44) -- (469,180.5) ;
\draw    (407.93,351.33) -- (469,294.39) ;
\draw [color={rgb, 255:red, 208; green, 2; blue, 27 }  ,draw opacity=1 ]   (286.2,305.94) -- (347.27,249) ;
\draw  [color={rgb, 255:red, 74; green, 144; blue, 226 }  ,draw opacity=1 ] (328.37,198) -- (450.5,198) -- (450.5,311.89) -- (328.37,311.89) -- cycle ;
\draw    (328.37,311.89) .. controls (403.8,299.6) and (425,278.75) .. (441.5,223.75) ;
\draw    (341,262.8) .. controls (403.13,255.66) and (444,186.75) .. (441.5,223.75) ;
\draw    (328.37,311.89) .. controls (357.4,286) and (308.2,265.2) .. (341,262.8) ;
\draw [color={rgb, 255:red, 74; green, 144; blue, 226 }  ,draw opacity=1 ]   (341,198.25) -- (341.5,310.75) ;
\draw [color={rgb, 255:red, 74; green, 144; blue, 226 }  ,draw opacity=1 ]   (357.5,198) -- (358,311.25) ;
\draw [color={rgb, 255:red, 74; green, 144; blue, 226 }  ,draw opacity=1 ]   (371,198) -- (372,310.25) ;
\draw [color={rgb, 255:red, 74; green, 144; blue, 226 }  ,draw opacity=1 ]   (385.93,198.82) -- (386.93,311.07) ;
\draw [color={rgb, 255:red, 74; green, 144; blue, 226 }  ,draw opacity=1 ]   (399.5,199) -- (400,312.25) ;
\draw [color={rgb, 255:red, 74; green, 144; blue, 226 }  ,draw opacity=1 ]   (414.43,199.32) -- (415.43,311.57) ;
\draw [color={rgb, 255:red, 74; green, 144; blue, 226 }  ,draw opacity=1 ]   (427.43,198.82) -- (428.43,311.07) ;
\draw [color={rgb, 255:red, 74; green, 144; blue, 226 }  ,draw opacity=1 ]   (440.43,198.82) -- (441.43,311.07) ;

\draw (399.33,153.73) node [anchor=north west][inner sep=0.75pt]    {$\mathcal{L}_{Y}$};
\draw (474.47,229.59) node [anchor=north west][inner sep=0.75pt]    {$B+F$};
\draw (250.8,289.6) node [anchor=north west][inner sep=0.75pt]  [color={rgb, 255:red, 208; green, 2; blue, 27 }  ,opacity=1 ]  {$\gamma _{\mathcal{L}_{E}}^{\lor }$};
\draw (328.03,265) node  [color={rgb, 255:red, 208; green, 2; blue, 27 }  ,opacity=1 ]  {$\bullet$};
\draw (415.03,244.5) node  [color={rgb, 255:red, 65; green, 117; blue, 5 }  ,opacity=1 ]  {$\bullet$};
\draw (302.5,177.4) node [anchor=north west][inner sep=0.75pt]  [font=\small]  {$\mathbb{\textcolor[rgb]{0.29,0.41,0.89}{F}}\textcolor[rgb]{0.29,0.41,0.89}{_{1}}$};
\draw (291.33,242.9) node [anchor=north west][inner sep=0.75pt]    {$\mathcal{L}_{E}$};

\end{tikzpicture}

    \caption{A local relative invariant of $\mathbb{F}_1$ (blue) in curve class $B+F$ passing through 1 fixed point (red) on a fiber $F \subset \mathbb{F}_1$ coming from the incidence condition $\gamma_{\cL_E}^{\vee}$ and an interior fixed point (green).}
    \label{fig:local_relative_F1}
\end{figure}
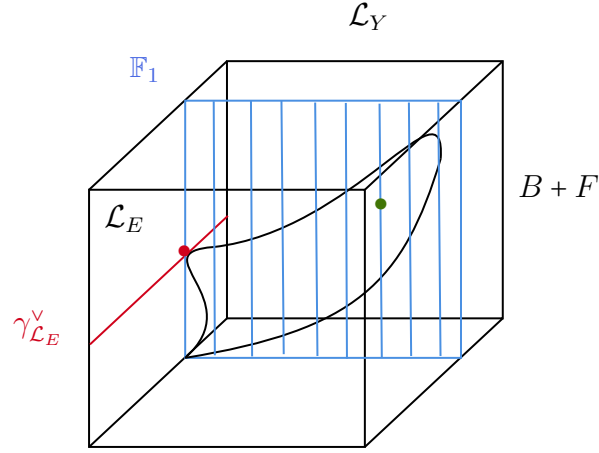

We show that $R_{g_C,2}(\FF_1(\log F), B+F)$ can be computed by the relationship between $q$-refined tropical curve counting and logarithmic Gromov-Witten theory of toric surfaces \cite{Bou} in the following,

\begin{proposition}
\label{prop:vertexC}
We have,

\[
\sum_{g \geq 0}R_{g, 2}(\FF_1(\log F), B+F) \hbar^{2g} = (-i)(q^{\frac{1}{2}} - q^{\frac{-1}{2}})
\]
where $q = e^{i\hbar}$. The first few terms on the right are $1 - \frac{1}{24}\hbar^2 + \frac{1}{1920}\hbar^4 - \frac{1}{322560}\hbar^6+\ldots$.
\end{proposition}
The proof of Proposition \ref{prop:vertexC} is given below Lemma \ref{lem:cos}. Recall that $B+F = \pi^*H$, and consider the following diagram,

\begin{center}
    \begin{tikzcd}
    \overline{\mathcal{M}}_{g,4}(\FF_1(\log \partial \FF_1), \pi^*H) \arrow[r, "\pi"] \arrow[d, "G"] & \barM_{g,4}(\bP^2(\log \partial \bP^2), H) \\
     \overline{\mathcal{M}}_{g,2}(\FF_1(\log F), \pi^*H)
    \end{tikzcd}
\end{center}
Denote $\barM_{g,4}(\bP^2(\log \partial \bP^2), H)$ to be the moduli space of genus-$g$, basic stable log maps to $\bP^2(\log \partial \bP^2)$ in the hyperplane class $H$, with 1 interior marked point and 3 relative marked points that each intersect distinct toric divisors of $\partial \bP^2$ with contact order 1. Denote $\overline{\mathcal{M}}_{g,4}(\FF_1(\log \partial \FF_1), \pi^*H)$ to be the moduli space of genus-$g$, basic stable log maps to $\FF_1(\log \partial \FF_1)$ in the hyperplane class $\pi^*H$ with 1 interior marked point, 1 relative marked point that intersects $\pi^*H$ once, and 2 relative marked points that intersect $F$ with contact order 1. Denote $\overline{\mathcal{M}}_{g,2}(\FF_1(\log F), \pi^*H)$ to be the moduli space of genus-$g$, basic stable log maps to $\FF_1(\log F)$ in the hyperplane class $\pi^*H$, with 1 interior marked point, and 2 relative marked points to the fiber $F$. We refer to \cite{GS13} \cite{Bou} \cite{Man1} for more details on the construction of moduli spaces of stable log maps.

We label the moduli spaces as,
\begin{equation}
\label{fig:logdiagram}
    \begin{tikzcd}
    A \arrow[r, "\pi"] \arrow[d, "G"] & B \\
     C 
    \end{tikzcd}
\end{equation}
The blow up map $\pi$ induces a map $A \rightarrow B$ via $f \mapsto \pi \circ f.$ Next, we define a birational morphism $G:A \rightarrow C$ that partially forgets the log structure of $\FF_1(\log \FF_1)$ and remembers a single fiber $F \subset \partial \FF_1$. Let $\cM_{(\FF_1, F)}$ be the divisorial log structure on $\FF_1$ given by $F$, and $\cM_{(\FF_1, \partial \FF_1)}$ the divisorial log structure on $\FF_1$ given by $\partial \FF_1.$ We have an inclusion of log structures $\cM_{(\FF_1, F)} \subset \cM_{(\FF_1, \partial \FF_1)}$, since $\cO^*_{\FF_1 \setminus F} \subset \cO^*_{\FF_1 \setminus \partial \FF_1}$, which induces a morphism of log schemes $\FF_1(\log \partial \FF_1) \rightarrow \FF_1(\log F)$ that is the identity map on underlying schemes. The morphism $G$ takes a stable log map $C \rightarrow \FF_1(\log \partial \FF_1)$ and composes it with $\FF_1(\log \partial \FF_1) \rightarrow \FF_1(\log F)$. In partially forgetting the log structure, the morphism $G$ forgets 2 relative marked points that define stable log maps in $A$ \cite{GS13}.

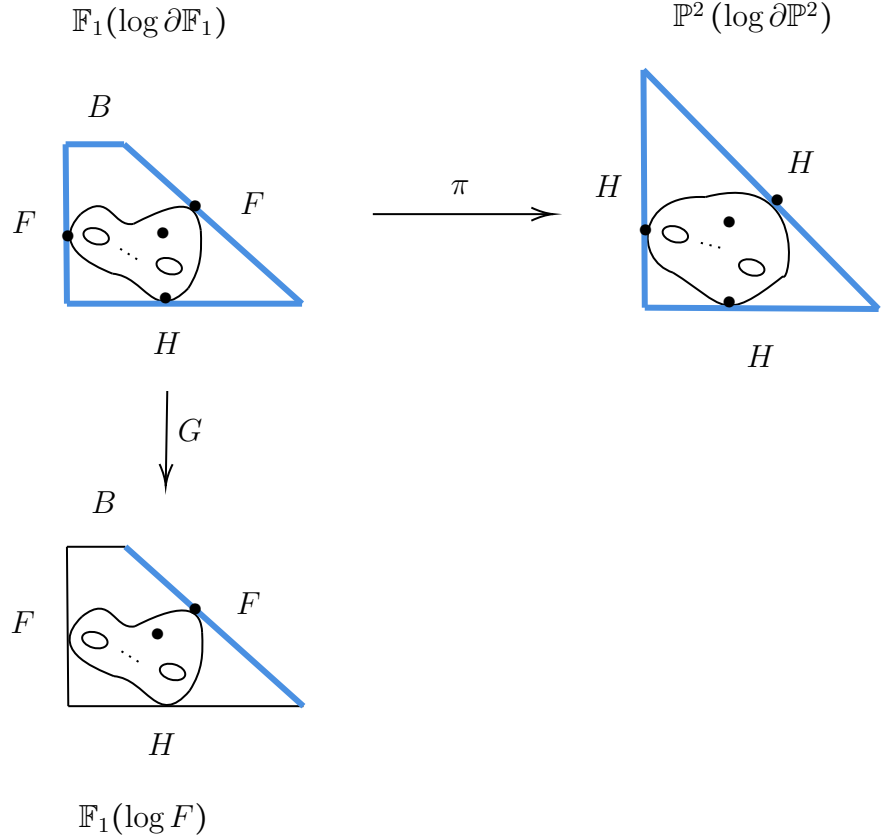
\begin{figure}[h!]
    \centering

\tikzset{every picture/.style={line width=0.75pt}} 

\begin{tikzpicture}[x=0.75pt,y=0.75pt,yscale=-1,xscale=1]

\draw    (234.84,123.06) -- (323.33,123.06) ;
\draw [shift={(325.33,123.06)}, rotate = 180] [color={rgb, 255:red, 0; green, 0; blue, 0 }  ][line width=0.75]    (10.93,-3.29) .. controls (6.95,-1.4) and (3.31,-0.3) .. (0,0) .. controls (3.31,0.3) and (6.95,1.4) .. (10.93,3.29)   ;
\draw    (131.68,211.82) -- (130.94,257.46) ;
\draw [shift={(130.91,259.46)}, rotate = 270.93] [color={rgb, 255:red, 0; green, 0; blue, 0 }  ][line width=0.75]    (10.93,-3.29) .. controls (6.95,-1.4) and (3.31,-0.3) .. (0,0) .. controls (3.31,0.3) and (6.95,1.4) .. (10.93,3.29)   ;
\draw    (104.25,152.95) .. controls (120.69,154.31) and (124.67,178) .. (142.19,160.29) ;
\draw    (142.19,160.29) .. controls (148.94,149.68) and (148.94,142.07) .. (148.51,131.19) ;
\draw    (114.79,127.92) .. controls (141.35,114.32) and (146,118) .. (148.51,131.19) ;
\draw    (104.25,152.95) .. controls (92.86,148.32) and (73.33,138.67) .. (88.23,123.03) ;
\draw    (114.79,127.92) .. controls (104.25,130.1) and (102.98,111.06) .. (88.23,123.03) ;
\draw    (401.73,114.5) .. controls (367.02,112.68) and (366,146) .. (385.48,152.41) ;
\draw    (401.73,114.5) .. controls (454,100.67) and (443.99,158.73) .. (440.74,154.22) ;
\draw    (385.48,152.41) .. controls (411.48,168.66) and (410,178.67) .. (440.74,154.22) ;
\draw  [dash pattern={on 0.84pt off 2.51pt}]  (108.04,140.44) -- (119.42,147.51) ;
\draw  [dash pattern={on 0.84pt off 2.51pt}]  (399.45,137.46) -- (409.92,140.1) ;
\draw [color={rgb, 255:red, 74; green, 144; blue, 226 }  ,draw opacity=1 ][line width=2.25]    (81.33,168) -- (199.33,168) ;
\draw [color={rgb, 255:red, 74; green, 144; blue, 226 }  ,draw opacity=1 ][line width=2.25]    (80.67,88) -- (81.33,168) ;
\draw [color={rgb, 255:red, 74; green, 144; blue, 226 }  ,draw opacity=1 ][line width=2.25]    (80.67,88) -- (110,88) ;
\draw [color={rgb, 255:red, 74; green, 144; blue, 226 }  ,draw opacity=1 ][line width=2.25]    (110,88) -- (199.33,168) ;
\draw    (104.91,354.95) .. controls (121.36,356.31) and (126.67,381.33) .. (142.86,362.29) ;
\draw    (142.86,362.29) .. controls (149.6,351.68) and (149.6,344.07) .. (149.18,333.19) ;
\draw    (115.45,329.92) .. controls (142.01,316.32) and (148.67,322) .. (149.18,333.19) ;
\draw    (104.91,354.95) .. controls (93.53,350.32) and (72.03,341.35) .. (88.89,325.03) ;
\draw    (115.45,329.92) .. controls (104.91,332.1) and (103.65,313.06) .. (88.89,325.03) ;
\draw  [dash pattern={on 0.84pt off 2.51pt}]  (108.71,342.44) -- (120.09,349.51) ;
\draw    (82,370) -- (200,370) ;
\draw    (81.33,290) -- (82,370) ;
\draw    (81.33,290) -- (110.67,290) ;
\draw [color={rgb, 255:red, 74; green, 144; blue, 226 }  ,draw opacity=1 ][fill={rgb, 255:red, 74; green, 144; blue, 226 }  ,fill opacity=1 ][line width=2.25]    (110.67,290) -- (200,370) ;
\draw [color={rgb, 255:red, 74; green, 144; blue, 226 }  ,draw opacity=1 ][line width=2.25]    (370.67,50.67) -- (371.33,170) ;
\draw [color={rgb, 255:red, 74; green, 144; blue, 226 }  ,draw opacity=1 ][line width=2.25]    (371.33,170) -- (488.67,170.67) ;
\draw [color={rgb, 255:red, 74; green, 144; blue, 226 }  ,draw opacity=1 ][line width=2.25]    (488.67,170.67) -- (370.67,50.67) ;
\draw   (126.38,147.42) .. controls (126.96,145.35) and (130.24,144.45) .. (133.71,145.42) .. controls (137.19,146.38) and (139.53,148.85) .. (138.95,150.92) .. controls (138.38,153) and (135.09,153.9) .. (131.62,152.93) .. controls (128.15,151.96) and (125.8,149.5) .. (126.38,147.42) -- cycle ;
\draw   (89.71,132.42) .. controls (90.29,130.35) and (93.57,129.45) .. (97.05,130.42) .. controls (100.52,131.38) and (102.87,133.85) .. (102.29,135.92) .. controls (101.71,138) and (98.43,138.9) .. (94.96,137.93) .. controls (91.48,136.96) and (89.14,134.5) .. (89.71,132.42) -- cycle ;
\draw   (380.38,131.42) .. controls (380.96,129.35) and (384.24,128.45) .. (387.71,129.42) .. controls (391.19,130.38) and (393.53,132.85) .. (392.95,134.92) .. controls (392.38,137) and (389.09,137.9) .. (385.62,136.93) .. controls (382.15,135.96) and (379.8,133.5) .. (380.38,131.42) -- cycle ;
\draw   (418.71,148.09) .. controls (419.29,146.02) and (422.57,145.12) .. (426.05,146.08) .. controls (429.52,147.05) and (431.87,149.52) .. (431.29,151.59) .. controls (430.71,153.67) and (427.43,154.56) .. (423.96,153.6) .. controls (420.48,152.63) and (418.14,150.17) .. (418.71,148.09) -- cycle ;
\draw   (89.05,335.09) .. controls (89.62,333.02) and (92.91,332.12) .. (96.38,333.08) .. controls (99.85,334.05) and (102.2,336.51) .. (101.62,338.59) .. controls (101.04,340.66) and (97.76,341.56) .. (94.29,340.6) .. controls (90.82,339.63) and (88.47,337.17) .. (89.05,335.09) -- cycle ;
\draw   (128.05,351.09) .. controls (128.62,349.02) and (131.91,348.12) .. (135.38,349.08) .. controls (138.85,350.05) and (141.2,352.51) .. (140.62,354.59) .. controls (140.04,356.66) and (136.76,357.56) .. (133.29,356.6) .. controls (129.82,355.63) and (127.47,353.17) .. (128.05,351.09) -- cycle ;

\draw (272.42,104.36) node [anchor=north west][inner sep=0.75pt]    {$\pi $};
\draw (135.51,224.46) node [anchor=north west][inner sep=0.75pt]    {$G$};
\draw (52.12,120) node [anchor=north west][inner sep=0.75pt]    {$F$};
\draw (166.81,110.43) node [anchor=north west][inner sep=0.75pt]    {$F$};
\draw (89.73,61.85) node [anchor=north west][inner sep=0.75pt]    {$B$};
\draw (122.99,180.64) node [anchor=north west][inner sep=0.75pt]    {$H$};
\draw (344.54,102.04) node [anchor=north west][inner sep=0.75pt]    {$H$};
\draw (441.18,90.07) node [anchor=north west][inner sep=0.75pt]    {$H$};
\draw (420.95,187.53) node [anchor=north west][inner sep=0.75pt]    {$H$};
\draw (82.7,17.36) node [anchor=north west][inner sep=0.75pt]    {$\mathbb{F}_{1}(\log \partial \mathbb{F}_{1})$};
\draw (85.76,417.1) node [anchor=north west][inner sep=0.75pt]    {$\mathbb{F}_{1}(\log F)$};
\draw (384.12,14.84) node [anchor=north west][inner sep=0.75pt]    {$\mathbb{P}^{2}\left(\log \partial \mathbb{P}^{2}\right)$};
\draw (51.46,320.66) node [anchor=north west][inner sep=0.75pt]    {$F$};
\draw (164.81,311.1) node [anchor=north west][inner sep=0.75pt]    {$F$};
\draw (92.06,260.85) node [anchor=north west][inner sep=0.75pt]    {$B$};
\draw (120.32,381.97) node [anchor=north west][inner sep=0.75pt]    {$H$};
\draw (413.79,127.68) node    {$\bullet$};
\draw (437.79,116.68) node    {$\bullet$};
\draw (413.79,167.68) node    {$\bullet$};
\draw (371.79,131.68) node    {$\bullet$};
\draw (129.45,133.35) node    {$\bullet$};
\draw (145.79,119.68) node    {$\bullet$};
\draw (130.79,165.68) node    {$\bullet$};
\draw (81.79,134.68) node    {$\bullet$};
\draw (126.79,334.68) node    {$\bullet$};
\draw (145.79,321.68) node    {$\bullet$};

\end{tikzpicture}

    \caption{The stable log maps with their marked points in Diagram \ref{fig:logdiagram}. The divisorial log structures are shaded in blue.}
    \label{fig:enter-label}
\end{figure}

To prove Proposition \ref{prop:vertexC}, we relate the virtual classes of the moduli spaces, and compute the virtual degree of $G$. We first recall the definition of torically transverse curves.

\begin{definition}[Definition 4.1 of \cite{NS}]
    Let $X$ be a toric variety. An algebraic curve $C \subset X$ is \textit{torically transverse or tt} if it is disjoint from all toric strata of codimension $>1.$ A stable map $\varphi:C \rightarrow X$ defined over a scheme $S$ is \textit{torically transverse or tt} if the following holds for the restriction $\varphi_s$ of $\varphi$ to every geometric point $s \rightarrow S: \varphi_s^{-1}(int X) \subset C_s$ is dense and $\varphi_s(C_s) \subset X$ is a torically transverse curve. 
\end{definition}

In \cite{NS}, they consider stable log maps to toric varieties with the toric log structure. For us, the stable log maps in $C$ have non-toric log structures. Hence, we define the locus of tt-curves of $C$ to be those stable log maps in which no component sinks into any component of the toric boundary. We denote $A^{tt}$ and $C^{tt}$ to be the open substacks of tt-curves in $A$ and $C$, respectively.

Let $G^{tt} := G|_{A^{tt}}$ be the restriction of the morphism $G$ onto $A^{tt}$. By Proposition 5.1, \cite{GS13}, $A$ carries a perfect obstruction theory $E_1^{\bullet}$ relative to the log stack of genus-g, pre-stable curves with 4 markings $\mathfrak{M}_{g,4}$, and defines a virtual fundamental class $[A]^{vir} \in A_*(A, \mathbb{Q})$. When restricted to $A^{tt}$, the resulting obstruction theory $E_1^{\bullet}|_{A^{tt}}$ remains perfect. Hence, $E_1^{\bullet}|_{A^{tt}}$ defines a cycle class of $A_*(A, \mathbb{Q})$, which is equal to $[A]^{vir}|_{A^{tt}}$. Hence, we define $[A^{tt}]^{vir} := [A]^{vir}|_{tt}.$ Similarly, let $E^{\bullet}_2$ be the perfect obstruction theory defined on $C$, and define $[C^{tt}]^{vir} := [C]^{vir}|_{tt}.$ The morphism $G$ is an isomorphism between $A^{tt}$ and $C^{tt}$, since there is only way to recover the 2 forgotten markings each mapping to distinct toric divisors. Hence, we have the equality of virtual classes, 

\begin{lemma}
\label{lem:cos}
    
    \[
    G^{tt}_*[A^{tt}]^{vir} = [C^{tt}]^{vir}
    \]
\begin{proof}
    We apply Theorem 5.0.1, \cite{Cos}. We have the diagram,
    
\begin{center}
\begin{tikzcd}[column sep=small]
E^{\bullet}_1|_{A^{tt}} \arrow[d] & & E^{\bullet}_2|_{C^{tt}} \arrow[d] \\
A^{tt} \arrow[dr, "ft_1"] \arrow[rr, "G^{tt}"] & & C^{tt} \arrow[dl, "ft_2"] \\
& \mathfrak{M}_{g,2} 
\end{tikzcd}
\end{center}
where $\mathfrak{M}_{g,2}$ is the log smooth stack of genus-g, pre-stable curves with 2 markings of pure dimension $3g-1$. The forgetful map $ft_1$ forgets the stable log map and 2 relative marked points mapping to $\pi^*H$ and $F$ respectively, and stabilizes. The forgetful map $ft_2$ forgets the stable log map and stabilizes. On the torically transverse locus, the curve class $B+F$ is transverse with generic contact order 1 with the toric divisors, and the logarithmic obstruction theories $E^{\bullet}_i$ restricted to tt-loci are both isomorphic to the respective obstruction theories of underlying stable map moduli spaces obtained by forgetting log structures. Thus, we have $(G^{tt})^* E^{\bullet}_2|_{C^{tt}} \cong E^{\bullet}_1|_{A^{tt}}$. The other assumptions for Theorem 5.0.1, \cite{Cos} are also satisfied. Since forgetting two marked points $\mathfrak{M}_{g,4} \rightarrow \mathfrak{M}_{g,2}$ is \'etale, the obstruction theory $E_1^{\bullet}|_{A^{tt}}$ relative to $\mathfrak{M}_{g,4}$ is isomorphic to the obstruction theory $E_1^{\bullet}|_{A^{tt}}$ relative to $\mathfrak{M}_{g,2}$. Hence, the obstruction theory $E_1^{\bullet}|_{A^{tt}}$ relative to $\mathfrak{M}_{g,2}$ also defines $[A^{tt}]^{vir}$. Thus, applying Theorem 5.0.1, \cite{Cos} to the diagram yields the desired equality. 
\end{proof}
\end{lemma}

\begin{proof}[Proof of Proposition \ref{prop:vertexC}]
\label{proof:vertexC}
Define the class $\gamma := (-1)^g \lambda_g ev_1^*([pt_1])ev_2^*([pt_2]) \in A^{g+3}(C)$ where $[pt_1] \in A^2(\FF_1)$ is an interior point and $[pt_2] \in A^1(F)$ is a point on a fiber $F \subset \partial \FF_1$ away from a toric fixed point. By definition, we have $R_{g,2}(\FF_1(\log F), \pi^*H) = [C]^{vir} \cap \gamma$.

By Lemma \ref{lem:irreducible}, we can assume curves in the class $\pi^*H$ are irreducible. Generic irreducible curves in the linear system $|\pi^*H|$ passing through $[pt_1]$ and $[pt_2]$ do not pass through a third toric fixed point, and hence they are torically transverse. We can therefore evaluate $[C]^{vir} \cap \gamma$ on the tt-locus $C^{tt}$, 

\[
[C]^{vir} \cap \gamma = [C^{tt}]^{vir} \cap \gamma
\]
and similarly we have $[A]^{vir} \cap \gamma = [A^{tt}]^{vir} \cap \gamma$. By Lemma \ref{lem:cos}, we have,

\[
[C]^{vir} \cap \gamma = G^{tt}_* [A^{tt}]^{vir} \cap \gamma
\]
By the projection formula, we have,

\[
[C]^{vir} \cap \gamma = [A^{tt}]^{vir} \cap (G^{tt})^* \gamma
\]
Since $G$ remembers the $\bP^1$-fiber $F$ that $[pt_2]$ lies on, we have $(G^{tt})^* \gamma = \gamma$. Hence, the above is,

\[
[C]^{vir} \cap \gamma = [A^{tt}]^{vir} \cap \gamma
\]
Define $\gamma' := (-1)^g \lambda_g ev_1^*([pt'_1])ev_2^*([pt'_2]) \in A^{g+3}(B)$ where $[pt'_1] \in A^2(\bP^2)$ is an interior point and $[pt'_2] \in A^1(\bP^1)$ is a point on $\bP^1 \subset \partial \bP^2$ away from a toric fixed point. We have $\gamma = \pi^*\gamma'$ since $[pt_2']$ is chosen away from a toric fixed point that is blown up. By Lemma \ref{lem:irreducible}, we have $[A^{tt}]^{vir} \cap \gamma = [A]^{vir} \cap \gamma$. By the projection formula, we then have,

\[
[C]^{vir} \cap \gamma = \pi_*[A]^{vir} \cap \gamma'
\]
By the birational invariance of log Gromov-Witten invariants or Theorem 1.1.1, \cite{AW}, we have $\pi_*[A]^{vir} = [B]^{vir}$. Hence,

\[
[C]^{vir} \cap \gamma = [B]^{vir} \cap \gamma'
\]
By Theorem 6 of \cite{Bou}, $[B]^{vir} \cap \gamma'$ is the genus-$g$, logarithmic Gromov-Witten invariant of $\bP^2(\log \partial \bP^2)$ with $\lambda_g$-class, fixing 1 interior point and 1 point on the toric boundary, and is the coefficient of $\hbar^{2g}$ in $(-i)(q^{\frac{1}{2}} - q^{\frac{-1}{2}})$.
\end{proof}

Proposition \ref{prop:vertexC} shows that one of the log invariants in the degeneration can be evaluated by $q$-refined tropical curve counting \cite{Bou}. In Appendix \ref{sec:evaluation_vertexC}, we also directly evaluate the genus-1 invariant associated to vertex $V_3$ to be $\frac{-1}{24}$, i.e. the $\hbar^2$-coefficient of $(-i)(q^{\frac{1}{2}} - q^{\frac{-1}{2}})$. Using the invariants defined in Sections \ref{eq:vertexA_invariant}, \ref{eq:vertexB_invariant}, \ref{eq:vertexC_invariant}, we have,

\begin{proposition}
\label{prop:deg}
We have,
   \begin{align*}
     N_{g,1}(Z, \beta+h) = \sum_{\substack{\Gamma \in \Gamma(g,n,\beta), \\g = g_{V_1} + g_{V_3}  + g_{V_3}}}
    &\Bigl[(\beta\cdot E - 1)R_{g_{V_1} ,(1, \beta\cdot E-1)}(X(\log E), \beta)\cdot \\
    &R_{g_{V_2}, (\beta\cdot E - 1)}(\FF_1(\log F), (\beta\cdot E-1)B)R_{g_{V_3}, 2}(\FF_1(\log F), B+F)\Bigr]  
   \end{align*}

   \begin{proof}
   We order the edges of the bipartite graph $\Gamma$ in \ref{thm:deg-graphs} such that the edge representing the fixed relative contact point is the top edge. Applying the degeneration formula of \cite{KLR} we obtain the desired formula.
    \end{proof}
\end{proposition}

\section{Obtaining higher genus local Gromov-Witten invariants}
\label{sec:proof-of-main}

Recall that $\pi:\widehat{X} \rightarrow X$ is the blow up at a point. In this section, we prove Theorem \ref{thm:main} relating higher genus invariants $N_{g,1}(Z, \beta+h)$ to local invariants of $\widehat{X}$. Let $N_{g,0}(K_{\widehat{X}}, \pi^*\beta-C)$ be the genus-0, unmarked local Gromov-Witten invariant $\widehat{X}$ in curve class $\pi^*\beta-C$ \cite{CKYZ}, and $n_{g,0}(K_{\widehat{X}}, \pi^*\beta-C)$ to be the corresponding genus-$g$ Gopakumar-Vafa invariant \cite{GV1} \cite{GV2}. Given an effective curve class $\beta \in NE(X)$, let $ e:= \beta\cdot E$. Define the following generating functions,
\begin{align*}
    F_{V_2} &= \sum_{g_{V_2} \geq 0}R_{g_{V_2},(e-1)}(\FF_1(\log F), (e-1)B)\hbar^{2g_{V_2}} \\
    F_{V_3} &= \sum_{g_{V_3} \geq 0}R_{g_{V_3}, 2}(\FF_1(\log F), B+F)\hbar^{2g_{V_3}}
\end{align*}
Note that $F_{V_2}$ and $F_{V_3}$ are independent of $\beta \in NE(X)$. Using Proposition \ref{prop:deg}, we prove Theorem \ref{thm:main_intro}, 

\begin{theorem}[= Theorem \ref{thm:main_intro}]
\label{thm:main}
There exists constants $c(g, \beta) \in \mathbb{Q}$ (Equation \ref{eq:proof-of-oP1,5}) such that,
\begin{multline*}
    \sum_{\substack{g \geq 0, \\ \beta \in NE(X)}}N_{g,1}(Z, \beta+h)\hbar^{2g}Q^{\beta} = \\ \sum_{\substack{g \geq 0, \\ \beta \in NE(X)}}\left[c(g, \beta)n_{g}\left(K_{\widehat{X}}, \pi^*\beta-C\right)\left(2\sin\frac{\hbar}{2}\right)^{2g - 2}Q^{\beta}\right] - \Delta
\end{multline*}
where the discrepancy $\Delta$ (Equation \ref{eq:delta_pl}) is expressed by the Gromov-Witten theory of $E$, and genus-$g$, 2-pointed logarithmic invariants $R_{g, (1,\beta\cdot E - 1)}(X(\log E), \beta)$ for all $g \geq 0$ and $\beta \in NE(X)$.
\begin{proof}

From here on, we simplify notation by indexing sums $\displaystyle{\sum_{g,\beta}}$ with \\$\beta \in NE(X)$ and $g \geq 0$, when there is no confusion. Summing over all genus in Proposition \ref{prop:deg},

\begin{equation}
\label{eq:proof-of-oP1,1}
    \begin{split}
        \sum_{g \geq 0}N_{g,1}(Z, \beta+h)\hbar^{2g} &= (\beta\cdot E - 1)\left(\sum_{g_{V_1} \geq 0}R_{g_{V_1},(1, \beta\cdot E-1)}(X(\log E), \beta)\hbar^{2g_{V_1}}\right)F_{V_2} F_{V_3}
    \end{split}
\end{equation}
Summing over all curve classes $\beta \in NE(X)$, and applying Corollary 6.6 of \cite{GRZ} to $R_{g_{V_1}, 2}(X(\log E))$ in Equation \ref{eq:proof-of-oP1,1},

\begin{equation}
\label{eq:proof-of-oP1,2}
    \begin{split}
        \sum_{g, \beta}N_{g,1}(Z, \beta+h)\hbar^{2g}Q^{\beta} &= \sum_{\beta \in NE(X)}\Biggl[(\beta\cdot E - 1)\biggl(\sum_{g_{V_1} \geq 0}\Bigl[R_{g_{V_1}, (\beta\cdot E - 1)}(\widehat{X}, \pi^*\beta-C) \\
        &- \sum_{i=0}^{g_{V_1} - 1} R_{i,(1, \beta\cdot E -1)}(X(\log E), \beta)N(g_{V_1}-i, 1)\Bigr]\hbar^{2g_{V_1}}\biggr)F_{V_2} F_{V_3}\Biggr]Q^{\beta}
    \end{split}
\end{equation}

For reasons we shall see, we define $\Delta$ to be the term,


\begin{equation}
\label{eq:delta_pl}
    \begin{split}
        \Delta &:= \sum_{\beta}\Biggl[(\beta\cdot E - 1)\Biggl(\sum_{g_{V_1} \geq 0}\biggl[(-1)^{\beta\cdot E}(\beta\cdot E - 1)\sum_{n \geq 0}\biggl[\sum_{\substack{g_{V_1} = h+g_1+\ldots+g_n, \\ \mathbf{a} = (a_1,\ldots, a_n) \in \mathbb{Z}^n_{\geq 0}, \\ \beta = d_E [E] + \beta_1 + \ldots + \beta_n, \\ d_E \geq 0, \beta_j \cdot D > 0}} \frac{(-1)^{g_{V_1}-1 + (E\cdot E)d_E}(E\cdot E)^m}{m!|Aut(\mathbf{a}, g_{V_1})|}  \\ 
        & N_{h, (\mathbf{a}, 1^m)}(E, d_E)\prod_{j=1}^n ((-1)^{\beta_j \cdot E}(\beta_j \cdot E) R_{g_j, (\beta_j\cdot E)}(\widehat{X}, \beta_j) \biggr] \\
        &+ \sum_{i=0}^{g_{V_1} - 1} R_{i,(1, \beta\cdot E -1)}(X(\log E), \beta)N(g_{V_1}-i, 1)\biggr]\hbar^{2g_{V_1}}\biggr)F_{V_2} F_{V_3} \Biggr]Q^{\beta}
    \end{split}
\end{equation}
where $N_{h, (\mathbf{a}, 1^m)}(E, d_E)$ are stationary invariants of $E$ defined in Section \ref{sec:g>0_loglocal}, Appendix. For $g_{V_1} \geq 0$ and $\beta \in NE(X)$, define $\Delta(g_{V_1}, \beta)$ by the sum, 

\begin{equation}
    \label{eq:delta_pl_gb}
    \Delta = \sum_{g_{V_1} \geq 0}\sum_{\beta \in NE(X)} \Delta(g_{V_1}, \beta) \hbar^{2g_{V_1}}Q^{\beta}
\end{equation}

Applying the $g>0$ log-local principle (combining Propositions 3.1 and 3.4 of \cite{BFGW}) to $R_{g_{V_1}}(\widehat{X}, \pi^*\beta-C)$ in Equation \ref{eq:proof-of-oP1,2}, we have,

\begin{equation}
\label{eq:proof-of-oP1,4}
    \begin{split}
        \sum_{g, \beta}N_{g,1}(Z, \beta+h)\hbar^{2g} &= \sum_{\beta}\left[(-1)^{\beta\cdot E}(\beta\cdot E - 1)^2\left(\sum_{g_{V_1} \geq 0}N_{g_{V_1}}(K_{\widehat{X}}, \pi^*\beta-C)\hbar^{2g_{V_1}}\right)F_{V_2} F_{V_3}\right]Q^{\beta} \\
        &- \Delta
    \end{split}
\end{equation}
For each $g_{V_1} \geq 0$ and $\beta \in NE(X)$, there exists a constant $c(g_{V_1}, \beta)$ that represents the overall contribution of $F_{V_2} F_{V_3}$ to the coefficient of $\hbar^{2g_{V_1}}$. Also absorbing $(-1)^{\beta\cdot E}(\beta\cdot E - 1)^2$ into $c(g_{V_1},\beta)$, Equation \ref{eq:proof-of-oP1,4} becomes,

\begin{equation}
\label{eq:proof-of-oP1,5}
    \begin{split}
        \sum_{g, \beta}N_{g,1}(Z, \beta+h)\hbar^{2g} &= \sum_{g, \beta}\left[c(g, \beta)N_{g}(K_{\widehat{X}}, \pi^*\beta-C)\hbar^{2g}Q^{\beta}\right] - \Delta
    \end{split}
\end{equation}
after relabelling $g_{V_1}$ by $g$. Substituting the closed Gopakumar-Vafa formula for toric Calabi-Yau threefolds \cite{GV1} \cite{GV2}, Equation \ref{eq:proof-of-oP1,5} becomes,

\[
    \sum_{g, \beta}N_{g,1}(Z, \beta+h)\hbar^{2g} = \sum_{g, \beta}\left[c(g, \beta)n_{g}(K_{\widehat{X}}, \pi^*\beta-C)\left(2\sin\frac{\hbar}{2}\right)^{2g - 2}Q^{\beta}\right] - \Delta
\]
\end{proof}
\end{theorem}

Theorem \ref{thm:main} relates higher genus invariants $N_{g,1}(Z, \beta+h)$ of the projectivized canonical bundle to higher genus Gopakumar-Vafa invariants of $K_{\widehat{X}}$. By Corollary 6.6 of \cite{GRZ}, the discrepancy term $\Delta$ in Theorem \ref{thm:main} is expressible by 2-pointed log invariants $R_{g,2}(X(\log E), \beta)$ and stationary invariants $N_{h, (\mathbf{a}, 1^m)}(E, d_E)$ of the elliptic curve for all $g,h$.

\begin{remark}
\label{rem:stationary_log}
  The stationary Gromov-Witten theory of $E$ is quasimodular as it is expressible by Eisenstein series \cite{OP}. The 2-pointed log invariants of $X(\log E)$ appear in Gross-Siebert mirror symmetry as structure constants of the mirror algebra of theta functions \cite{GS16}; they can be computed via $q$-refined tropical curve counting \cite{Gra}.
\end{remark}

\subsection{Genus-1}
\label{subsec:explicit g=1,2}
We specialize Theorem \ref{thm:main} to genus-1. For simplicity, we will at times suppress notation for the log structure or curve class by writing $R_{g,(p,q)}(X(\log E), \beta)$ as $R_{g,(p,q)}(X)$. Let $n_1(K_{\widehat{X}}, \pi^*\beta-C)$ be the genus-1,  Gopakumar-Vafa invariant of $K_{\widehat{X}}$ in class $\pi^*\beta-C.$

\begin{corollary}[Theorem \ref{thm:main} in genus-1]
\label{cor:maing1}
Let $\beta \in NE(X,\mathbb{Z})$. We have,

\[
N_{1,1}(Z, \beta+h) = n_1(K_{\widehat{X}}, \pi^*\beta-C) - \delta_1(\beta)
\]
where $\delta_1(\beta)$ is expressed by genus-1 Gromov-Witten invariants of $E$ and defined in Appendix A, Equation \ref{eq:delta1}. 
\end{corollary}

\begin{proof}

Define $e := \beta \cdot E.$ The genus-1 invariants in $F_{V_2}$ and $F_{V_3}$ (definition in top of Section \ref{sec:proof-of-main}) are $\frac{(-1)^e}{24}$ and $\frac{-1}{24}$, respectively. By Proposition \ref{prop:deg}, we have,

\begin{equation}
\label{eq:explicitg1,1}
    \begin{split}
        N_{1,1}(Z, \beta+h) &= (e-1)\left[\frac{(-1)^e}{(e-1)^2}R_{1, (1,e-1)}(X(\log E), \beta) \right.\\
        &+ \left.\frac{(-1)^e}{24}R_{0, (1,e-1)}(X(\log E), \beta) \right.\\
        &+ \left.\frac{(-1)^{e+1}}{24(e-1)^2}R_{0, (1,e-1)}(X(\log E), \beta)\right] 
    \end{split}
\end{equation}
Applying Corollary 6.6 of \cite{GRZ} to $R_{1, (1,e-1)}(X)$, 

\begin{equation}
\label{eq:explicitg1,2}
\begin{split}
N_{1,1}(Z, \beta+h) &= \frac{(-1)^e}{(e-1)}R_{1, (e-1)}(\widehat{X}) + \left(\frac{(-1)^{e+1}}{12(e-1)} + \frac{(-1)^e(e-1)}{24}\right)R_{0, (1,e-1)}(X)
\end{split}
\end{equation}
The genus-1, log-local principle \cite{BFGW} tells us that,

\[
R_{1, (e-1)}(\widehat{X}) = (-1)^e(e-1)\left[N_1(K_{\widehat{X}}) + \frac{(-1)^{e+1}(e-1)}{24}R_{0, (e-1)}(\widehat{X}) - \delta_1(\beta)\right]
\]
where $\delta_1(\beta)$ is defined in Equation \ref{eq:delta1} in the Appendix; applying it to Equation \ref{eq:explicitg1,2}, 

\begin{equation}
\label{eq:explicitg1,3}
\begin{split}
N_{1,1}(Z, \beta+h) &= N_1(K_{\widehat{X}}) + \frac{(-1)^{e+1}(e-1)}{24}R_{0, (e-1)}(\widehat{X}) - \delta_1(\beta) \\
&+ \left(\frac{(-1)^{e+1}}{12(e-1)} + \frac{(-1)^e(e-1)}{24}\right)R_{0, (1,e-1)}(X, \beta)
\end{split}
\end{equation}
The genus-1 closed Gopakumar-Vafa formula for Calabi-Yau threefolds for the class $\pi^*\beta-C$ states that,

\begin{equation}
\label{eq:explicitg1,4}
 N_1(K_{\widehat{X}}, \pi^*\beta-C) = n_1(K_{\widehat{X}}, \pi^*\beta-C) + \frac{1}{12}n_0(K_{\widehat{X}}, \pi^*\beta-C)   
\end{equation}
By Corollary 6.6 of \cite{GRZ} and the $g=0$ log-local principle \cite{vGGR}, we have the equalities,

\begin{equation}
\label{eq:explicitg1,5}
  R_{0, (1,e-1)}(X(\log E), \beta) = R_{0, (e-1)}(\widehat{X}(\log \pi^*E-C), \pi^*\beta-C) = (-1)^e(e-1)n_0(K_{\widehat{X}}, \pi^*\beta-C)  
\end{equation}
Applying Equations \ref{eq:explicitg1,4} and \ref{eq:explicitg1,5} to Equation \ref{eq:explicitg1,3}, we have,

\begin{equation}
\label{eq:explicitg1,6}
\begin{split}
N_{1,1}(Z, \beta+h) &= n_1(K_{\widehat{X}}) + \frac{(-1)^e}{12(e-1)}R_{0, (1,e-1)}(X) + \frac{(-1)^{e+1}(e-1)}{24}R_{0, (e-1)}(\widehat{X}) - \delta_1(\beta) \\
&+ \left(\frac{(-1)^{e+1}}{12(e-1)} + \frac{(-1)^e(e-1)}{24}\right)R_{0, (1,e-1)}(X, \beta)
\end{split}
\end{equation}
Since $R_{0, (1, e-1)}(X) = R_{0, (e-1)}(\widehat{X})$ by Corollary 6.6, \cite{GRZ}, we have,

\[
  N_{1,1}(Z, \beta+h) = n_1(K_{\widehat{X}}, \pi^*\beta-C) - \delta_1(\beta)  
\]
\end{proof}

\begin{remark}
    The genus-0 open-closed result of \cite{Cha} states $N_{0,1}(Z,\beta+h) = O_0(K_X/L, \beta+\beta_0, 1)$. On the otherhand, by the remarks in Section 2.2 of \cite{GRZ} and Theorem 1.1, \cite{LLW}, we have $O_0(K_X/L,\beta+\beta_0, 1) = N_0(K_{\widehat{X}}, \pi^*\beta-C)$. The genus-0, Gopakumar-Vafa formula states $N_0(K_{\widehat{X}}, \pi^*\beta-C) = n_0(K_{\widehat{X}}, \pi^*\beta-C)$. Combining the above equalities, we have $N_{0,1}(Z, \beta+h) = n_0(K_{\widehat{X}}, \pi^*\beta-C)$, which is Theorem \ref{thm:main} specialized to genus-0. Corollary \ref{cor:maing1} tells us that in genus-1, the relation between $N_{1,1}(Z, \beta+h)$ and $n_0(K_{\widehat{X}}, \pi^*\beta-C)$ is corrected by stationary invariants of the elliptic curve $\delta_1(\beta)$.
\end{remark}

\section{Blow up formulas for projective bundles}

\label{sec:blow-up}

Genus-0 blow up formulas in Gromov-Witten theory were studied by \cite{Gat} \cite{Hu}. In real dimension-6, blow up formulas for descendant invariants were proven in all-genus \cite{HHKQ}. For logarithmic Gromov-Witten theory, blow up formulas appear in the work of \cite{AW}, which relates virtual classes of stable log moduli spaces for in fact more general birational morphisms.

In this section, we prove a blow up formula for invariants associated to projective bundles in all-genus in Theorem \ref{thm:blow-up}. We use the invariance of Gromov-Witten invariants under flops of threefolds \cite{LR}. We give a genus-1 formula of Theorem \ref{thm:blow-up} in Corollary \ref{cor:blow-up12}.

\subsection{Spaces involved}
Recall that we have a smooth log Calabi-Yau pair consisting of a smooth Fano surface $X$ with a smooth elliptic curve $E$, with $Z := \bP(K_X \oplus \cO_X)$. There are two distinguished sections $E_0, E_{\infty} \subset Z$ corresponding to the summands $\bP(0 \oplus \cO_X)$ and $\bP(K_X \oplus 0)$, respectively. Let $\pi:\widehat{X} \rightarrow X$ be the blow up at a single point of $X$, with exceptional curve $C.$ Define $\widehat{Z} := \bP(K_{\widehat{X}} \oplus \cO_{\widehat{X}}).$ Let $p \in E_{\infty}$, and $L \cong \bP^1 \subset Z$ be the unique fiber passing through $p$. Define $W := Bl_p Z$ to be the blow up at $p$ of $Z$, with the map $\pi_1: W \rightarrow Z$. Let $\tilde{L}$ be the strict transform of $L$ under $\pi_1$. We see that $\tilde{L}$ is a smooth rational curve with normal bundle $\cO_{\bP^1}(-1) \oplus \cO_{\bP^1}(-1)$.

\subsection{The invariants}

We relate 1-pointed Gromov-Witten invariants $N_{g,1}(Z)$ to unmarked invariants of $W$ via the intermediate space $\widehat{Z}.$ Recall that $N_{g,0}(K_{\widehat{X}}, \pi^*\beta-C)$ is the genus-0, local Gromov-Witten invariant of $\widehat{X}$ in curve class $\pi^*\beta-C$ \cite{CKYZ}. Let $\barM_{g,0}(W, \beta+\tilde{L})$ be the moduli space of genus-$g$ stable maps to $W$ in curve class $\beta+\tilde{L}$. Since $c_1(TW)(\beta) = c_1(TW)(\tilde{L}) = 0$, its virtual dimension is 0. We define,

\[
N_{g,0}(W, \beta+\tilde{L}) := \int_{[\barM_{g,0}(W, \beta+\tilde{L})]^{vir}} 1
\]

\subsection{Flop invariance}

Flops of 3-folds are birational transformations that are compositions of blow ups and blow downs along a rational curve with normal bundle $\cO_{\bP^1}(-1) \oplus \cO_{\bP^1}(-1)$, i.e. an exceptional curve. For $X = \bP^2$, refer to Figure \ref{fig:flop} for a flop between $\widehat{Z}$ and $W$ described in their toric fans. The following theorem states that Gromov-Witten invariants are functorial with respect to flops,

\begin{theorem}
\label{thm:flop}
(\cite{LR}) For a simple flop $\varphi: X \dashedrightarrow Y$ between threefolds, if $\beta$ is not a multiple of an exceptional curve, then for all $g \geq 0$,

\[
\int_{[\barM_{g,n}(X, \beta)]^{vir}} \prod_{i=1}^n ev_i^* \varphi^* \gamma_i = \int_{[\barM_{g,n}(Y, \varphi(\beta))]^{vir}} \prod_{i=1}^n ev_i^* \gamma_i
\]
where $\gamma_i \in H^*(Y).$
\end{theorem}

\begin{lemma}
\label{lem:blow-up1}
    For all $g \geq 0$, 
    
    \[
    N_{g,0}(W, \beta+\tilde{L}) = N_{g,0}(\widehat{Z}, \pi^*\beta - C)
    \]
    
    \begin{proof}
    From Proposition 3.1, \cite{LLW}, there exists a flop $\varphi: W \dashrightarrow \widehat{Z}$ along an exceptional curve $\tilde{L} \subset W$ such that $\varphi(\tilde{L}) = -C$. Then apply Theorem \ref{thm:flop}. 
    \end{proof}
\end{lemma}

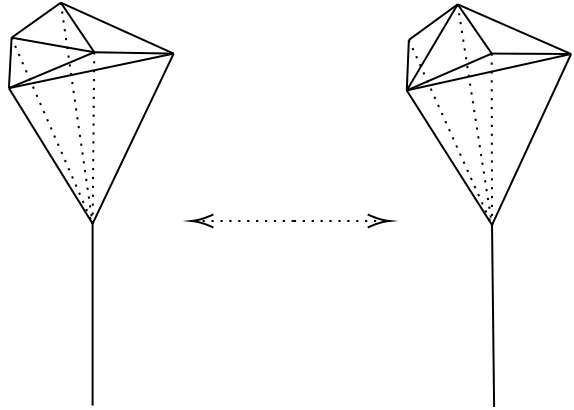
\begin{figure}[h!]
    \centering
    \tikzset{every picture/.style={line width=0.75pt}} 

\begin{tikzpicture}[x=0.75pt,y=0.75pt,yscale=-1,xscale=1]

\draw    (201.33,101.67) -- (244.16,83.65) ;
\draw    (244.16,83.65) -- (284,84.33) ;
\draw    (244.16,83.65) -- (227.38,58.76) ;
\draw    (201.33,101.67) -- (284,84.33) ;
\draw    (227.38,58.76) -- (284,84.33) ;
\draw  [dash pattern={on 0.84pt off 2.51pt}]  (244.16,83.65) -- (243.33,169.67) ;
\draw  [dash pattern={on 0.84pt off 2.51pt}]  (227.38,58.76) -- (243.33,169.67) ;
\draw    (201.33,101.67) -- (243.33,169.67) ;
\draw    (284,84.33) -- (243.33,169.67) ;
\draw    (243.33,169.67) -- (243.33,261) ;
\draw    (202.67,76.33) -- (244.16,83.65) ;
\draw    (202.67,76.33) -- (227.38,58.76) ;
\draw    (201.33,101.67) -- (202.67,76.33) ;
\draw  [dash pattern={on 0.84pt off 2.51pt}]  (202.67,76.33) -- (243.33,169.67) ;
\draw    (401,102.81) -- (443.62,84.31) ;
\draw    (443.62,84.31) -- (483.46,84.53) ;
\draw    (443.62,84.31) -- (426.55,59.61) ;
\draw    (401,102.81) -- (483.46,84.53) ;
\draw    (426.55,59.61) -- (483.46,84.53) ;
\draw  [dash pattern={on 0.84pt off 2.51pt}]  (443.62,84.31) -- (443.78,170.33) ;
\draw  [dash pattern={on 0.84pt off 2.51pt}]  (426.55,59.61) -- (443.78,170.33) ;
\draw    (401,102.81) -- (443.78,170.33) ;
\draw    (483.46,84.53) -- (443.78,170.33) ;
\draw    (443.78,170.33) -- (444.83,261.65) ;
\draw    (401,102.81) -- (426.55,59.61) ;
\draw    (402.04,77.47) -- (426.55,59.61) ;
\draw    (401,102.81) -- (402.04,77.47) ;
\draw  [dash pattern={on 0.84pt off 2.51pt}]  (402.04,77.47) -- (443.78,170.33) ;
\draw  [dash pattern={on 0.84pt off 2.51pt}]  (344.33,168.33) -- (390.33,168.33) ;
\draw [shift={(392.33,168.33)}, rotate = 180] [color={rgb, 255:red, 0; green, 0; blue, 0 }  ][line width=0.75]    (10.93,-3.29) .. controls (6.95,-1.4) and (3.31,-0.3) .. (0,0) .. controls (3.31,0.3) and (6.95,1.4) .. (10.93,3.29)   ;
\draw  [dash pattern={on 0.84pt off 2.51pt}]  (344.33,168.33) -- (295,168.33) ;
\draw [shift={(293,168.33)}, rotate = 360] [color={rgb, 255:red, 0; green, 0; blue, 0 }  ][line width=0.75]    (10.93,-3.29) .. controls (6.95,-1.4) and (3.31,-0.3) .. (0,0) .. controls (3.31,0.3) and (6.95,1.4) .. (10.93,3.29);
\end{tikzpicture}
    \caption{Flop between $\widehat{Z} = \bP(K_{\FF_1}\oplus\cO_{\FF_1})$ and $W = Bl_p \bP(K_{\bP^2}\oplus \cO_{\bP^2})$, whose toric fans (completed to convex fans) are the left and right, respectively.}
    \label{fig:flop}
\end{figure}

\subsection{Proof of Theorem \ref{thm:main_blow-up}}

By blowing up at a point, we effectively rid of the single point constraint defining $N_{g,1}(Z, \beta+h)$ (\ref{eq:def_of_invariant}) by equating them with invariants of $W = Bl_p Z$ with one less point constraint. We first have the following lemma,
    
\begin{lemma}
\label{lem:blow-up2}
    For all $g \geq 0$, we have that $N_{g,0}(K_{\widehat{X}}, \pi^*\beta - C) = N_{g,0}(\widehat{Z}, \pi^*\beta - C)$.
        \begin{proof}
        Under the $\mathbb{C}^*$-action that scales the $\bP^1$-fiber, the fixed point set of $\barM_{g,0}(\widehat{Z}, \pi^*\beta-C)$ is isomorphic to $\barM_{g,0}(\widehat{X}, \pi^*\beta-C)$. By virtual localization, the two invariants are equal (see Proposition 2.2, \cite{KM}).
        \end{proof}
\end{lemma}

\begin{theorem}[= Theorem \ref{thm:main_blow-up}]
\label{thm:blow-up}
    Let $W := Bl_p Z$ be the blow up of $Z$ at a point $p$ on its infinity section. 
    
    \[
    \sum_{\substack{g \geq 0, \\ \beta \in NE(X)}}N_{g,1}(Z, \beta+h)\hbar^{2g}Q^{\beta} =\sum_{\substack{g \geq 0, \\ \beta \in NE(X)}}\left[c(g, \beta)N_{g,0}(W, \beta + \tilde{L})\hbar^{2g}Q^{\beta}\right] - \Delta
    \]
    where $c(g,\beta) \in \mathbb{Q}$ and the $\Delta$ are given in Theorem \ref{thm:main}. 
    \begin{proof}
    Applying the Gopakumar-Vafa formula to the right side of Theorem \ref{thm:main}, there exist constants $c(g,\beta) \in \mathbb{Q}$ such that,
    
    \begin{equation}
    \label{eq:blow-up1}
        \begin{split}
           \sum_{\substack{g \geq 0, \\ \beta \in NE(X)}}N_{g,1}(Z, \beta+h)\hbar^{2g}Q^{\beta} &= \sum_{\substack{g \geq 0, \\ \beta \in NE(X)}}c(g, \beta)N_{g,0}(K_{\widehat{X}}, \pi^*\beta-C)\hbar^{2g}Q^{\beta} - \Delta 
        \end{split}
    \end{equation}
    By Lemma \ref{lem:blow-up2}, Equation \ref{eq:blow-up1} becomes,
    
    \begin{equation}
    \label{eq:blow-up2}
        \begin{split}
           \sum_{\substack{g \geq 0, \\ \beta \in NE(X)}}N_{g,1}(Z, \beta+h)\hbar^{2g}Q^{\beta}&= \sum_{\substack{g \geq 0, \\ \beta \in NE(X)}}c(g, \beta)N_{g,0}(\widehat{Z}, \pi^*\beta-C)\hbar^{2g}Q^{\beta} - \Delta
        \end{split}
    \end{equation}
    By Lemma \ref{lem:blow-up1}, Equation \ref{eq:blow-up2} becomes,
    
    \begin{equation}
        \begin{split}
           \sum_{\substack{g \geq 0, \\ \beta \in NE(X)}}N_{g,1}(Z, \beta+h)\hbar^{2g}Q^{\beta} &= \sum_{\substack{g \geq 0, \\ \beta \in NE(X)}}c(g, \beta)N_{g,0}(W, \beta + \tilde{L})\hbar^{2g}Q^{\beta} - \Delta
        \end{split}
    \end{equation}
\end{proof}
\end{theorem}

Theorem \ref{thm:blow-up} gives us blow up formula for invariants of projective bundles in all-genus. It also extends Theorem 1.1 of \cite{HHKQ} to include invariants defined by a single point constraint.

\subsection{Formula in genus-1}

\begin{corollary}
\label{cor:blow-up12}
In genus-1, Theorem \ref{thm:blow-up} is,

\[
\sum_{\beta \in NE(X)}N_{1,1}(Z, \beta+h)Q^{\beta} = \sum_{\beta \in NE(X)}\left[N_{1,0}(W, \beta + \tilde{L}) - \frac{1}{12}N_{0,0}(W, \beta + \tilde{L}) - \delta_1(\beta)\right]Q^{\beta}
\]
where $\delta_1(\beta)$ is defined in Equation \ref{eq:delta1}, Appendix.
\begin{proof}
We have,

\begin{align*}
    N_{1,1}(Z, \beta+h) &= n_{1,0}(K_{\widehat{X}}) - \delta_1(\beta) \\
    &= N_{1,0}(K_{\widehat{X}}, \pi^*\beta - C) - \frac{1}{12}N_{0,0}(K_{\widehat{X}}, \pi^*\beta - C) - \delta_1(\beta) \\
    &= N_{1,0}(\widehat{Z}, \pi^*\beta - C) - \frac{1}{12}N_{0,0}(\widehat{Z}, \pi^*\beta - C) - \delta_1(\beta) \\
    &= N_{1,0}(W, \beta+\tilde{L}) - \frac{1}{12}N_{0,0}(W, \beta+\tilde{L}) - \delta_1(\beta)
\end{align*}
where the 1st equality is Corollary \ref{cor:maing1}, the 2nd equality is the genus-1 Gopakumar-Vafa formula, the 3rd equality is Lemma \ref{lem:blow-up2}, and the 4th equality is Lemma \ref{lem:blow-up1}. Summing over all curve classes, we have the desired expression. 
\end{proof}
\end{corollary}

\section{All-genus open-closed correspondence for projective bundles}

\label{sec:op}

For this section, let $X$ additionally be toric, and $K_X$ be the toric canonical bundle. We prove an all-genus open-closed correspondence for projective bundles on smooth log Calabi-Yau pairs. 

\begin{theorem}[= Theorem \ref{thm:intro_op}]
\label{thm:op}
    Suppose that $X$ is a toric Fano surface, with $Z = \bP(K_X \oplus \cO_X)$ and $K_X$ the toric canonical bundle. 
    
    \begin{multline*}
   \sum_{\substack{g \geq 0, \\ \beta \in NE(X)}} N_{g,1}(Z, \beta+h)\hbar^{2g}Q^{\beta} = \\\sum_{\substack{g \geq 0, \\ \beta \in NE(X)}}\left[(-1)^{g+1}c(g, \beta)n^{open}_{g}(K_X/L, \beta+\beta_0, 1)\left(2\sin\frac{\hbar}{2}\right)^{2g - 2}Q^{\beta}\right] - \Delta
    \end{multline*}
    where $c(g,\beta) \in \mathbb{Q}$ and $\Delta$ are as in Theorem \ref{thm:main}.
\end{theorem}
Theorem \ref{thm:op} relies on Corollary 5.5 of \cite{GRZZ}, which proves an equality between the genus-$g$, winding-1, 1-holed open-BPS invariant $n_g^{open}(K_X/L, \beta+\beta_0, 1)$ of an outer AV-brane in framing-0 and the genus-$g$, closed Gopakumar-Vafa invariant $n_g(K_{\widehat{X}}, \pi^*\beta-C)$ in class $\pi^*\beta-C$. 
\begin{corollary}[Corollary 5.5, \cite{GRZZ}]
\label{cor:winding-1}
    Suppose $X$ is toric, and $\pi:\widehat{X} \rightarrow X$ a toric blow up at a point with exceptional curve $C$. Let $K_{X}, K_{\widehat{X}}$ be the toric canonical bundles. 
    
    \[
    n_g(K_{\widehat{X}}, \pi^*\beta-C) = (-1)^{g+1}n_g^{open}(K_X/L, \beta+\beta_0, 1)
    \]
\end{corollary} 
Corollary \ref{cor:winding-1} follows from flop invariance and glueing formula of the Topological Vertex \cite{AKMV}. A computational verification using the Topological Vertex for Corollary \ref{cor:winding-1} is provided in various degrees and all-genus in Section 5.5, \cite{GRZZ}. 

\begin{proof}[Proof of Theorem \ref{thm:op}]
    Apply Corollary \ref{cor:winding-1} to Theorem \ref{thm:main}.
\end{proof}

\section{Appendix}

\subsection{The $g>0$ log-local principle}
\label{sec:g>0_loglocal}

The higher genus log-local principle of \cite{BFGW} expresses the genus-$g$ maximal tangency, logarithmic Gromov-Witten invariants of $X(\log E)$ in terms of local Gromov-Witten invariants of $K_X$ and the stationary Gromov-Witten theory of the elliptic curve $E.$ In this section, we specialize the $g>0$ log-local principle to genus-1, and explain the terms $\Delta$
and $\delta_1(\beta)$ used in Theorem \ref{thm:main} and Corollary \ref{cor:maing1}, respectively.

\begin{theorem}[$g>0$ log-local principle of \cite{BFGW}]
\label{thm:g>0log-local}
    For every $g \geq 0$, we have that,
    
    \[
    F_g^{K_X} = (-1)^g F_g^{X/E} + \sum_{n\geq 0}\sum_{\substack{g = h+g_1+\ldots+g_n, \\ \mathbf{a} = (a_1,\ldots, a_n) \in \mathbb{Z}_{\geq 0}^n \\ (a_j, g_j) \neq (0,0), \sum_{j=1}^n a_j = 2h-2}}\frac{(-1)^{h-1}F_{h,\mathbf{a}}^E}{|Aut(\mathbf{a},\mathbf{g})|} \prod_{j=1}^n (-1)^{g_j-1}D^{a_j+2}F_{g_j}^{X/E}
    \]
    Equivalently, for a curve class $\beta$ satisfying $\beta\cdot E > 0$, we have,
    
    \begin{align*}
    N_{g}(K_X, \beta) &=  \frac{(-1)^{\beta\cdot E-1}}{\beta\cdot E}R_g(X(\log E), \beta) + \sum_{n \geq 0}\sum_{\substack{g = h+g_1+\ldots+g_n, \\ \mathbf{a} = (a_1,\ldots, a_n) \in \mathbb{Z}^n_{\geq 0}, \\ \beta = d_E [E] + \beta_1 + \ldots + \beta_n, \\ d_E \geq 0, \beta_j \cdot D > 0}}\biggl[\frac{(-1)^{g-1 + (E\cdot E)d_E}(E\cdot E)^m}{m!|Aut(\mathbf{a}, g)|} \\
    &N_{h, (\mathbf{a}, 1^m)}(E, d_E)\prod_{j=1}^n \left((-1)^{\beta_j \cdot E}(\beta_j \cdot E) R_{g_j, (\beta_j\cdot E)}(X, \beta_j)\right)\biggr]
    \end{align*}
where $m := 2g-2 - \sum_j a_j$, and $|Aut(\mathbf{a}, g)| = |Aut(a_1,g_1)|\ldots |Aut(a_n, g_n)|$ with $|Aut(a_i, g_i)|$ being the number of partitions of $a_i$ into $g_i$ boxes.
\end{theorem}

The stationary invariants $N_{h, (\textbf{a}, 1^m)}(E, d_E)$ of the elliptic curve for $\textbf{a} \in \mathbb{Z}^n_{\geq 0}$ are defined as,

\begin{align*}
   N_{h, (\textbf{a}, 1^m)}(E, d_E) := \int_{[\barM_{h,n+m}(E, d_E)]^{vir}}\prod_{i=1}^n ev_i^*[pt] \psi_i^{a_i} \prod_{j=1}^m ev_j^*[pt] \psi_j 
\end{align*}
where $\psi_j \in H^2(\barM_{h,n+m}(E, d_E))$ is the $\psi-$insertion at the $j$-th marked point, and $[pt] \in H^2(E)$ is the Poincare-dual of a point on $E.$

\subsubsection{Genus-1}

We specialize Theorem \ref{thm:g>0log-local} to genus-1. The genus-1 generating series are,

\begin{align*}
    F_1^{K_X} &:= \left(\frac{1-\delta_{(E\cdot E), 0}\chi(X)}{(E\cdot E)24} - \frac{1}{24}\right)\log Q^E + \sum_{\beta | \beta\cdot E > 0}N_1(K_X, \beta) Q^{\beta} \\
    F_1^{X/E} &:= -\frac{1-\delta_{(E\cdot E), 0}\chi(X)}{(E\cdot E)24}\log Q^E + \sum_{\beta | \beta\cdot E > 0}\frac{(-1)^{\beta\cdot E}}{\beta\cdot E} R_1(X(\log D), \beta) Q^{\beta} \\
    F_{1,0^n}^E &:= \delta_{n,0}\frac{-1}{24}\log\left((-1)^{E\cdot E}\tilde{\mathcal{Q}}\right) + \sum_{d \geq 0}\tilde{\mathcal{Q}}^d\int_{[\overline{\mathcal{M}}_{1,n}(E,d)]^{vir}}\prod_{i=1}^n ev_i^*([pt])
\end{align*}
\begin{remark}
    The closed string symplectic parameter $Q$ keeps track of effective curve classes $\beta \in NE(X)$. It is related by mirror symmetry to the closed complex parameter $q$ on the stringy K\"ahler moduli space associated to $K_X.$ The variable $\tilde{\mathcal{Q}} = \tilde{\mathcal{Q}}(q)$ is related to $Q$ by the change of variables,

\begin{equation}
\label{eq:var_change}
  \tilde{\mathcal{Q}} = (-1)^{E\cdot E}Q^E\exp\left(\sum_{\beta|\beta\cdot E>0}(-1)^{\beta\cdot E}(\beta\cdot E)R_0(X(\log E),\beta)Q^{\beta}\right) = (-1)^{E\cdot E}\exp\left(-D^2 F_0^{K_X}\right)  
\end{equation}
\end{remark} 

In Theorem \ref{thm:g>0log-local}, if $h=1$, then $2h-2 = 0$ and $a_j = 0$ for all $j$. Because $(a_j, g_j) \neq (0,0)$, the only remaining term in the sum is $F_{1, 0^n}^E$. Hence, we have,

\begin{align*}
    F_1^{K_X} &= -F_1^{X/E} + \sum_{n \geq 0}F_{1,0^n}^E
\end{align*}
The virtual dimension of $\overline{\mathcal{M}}_{1,n}(E,d)$ is $n$, hence only the generating series $F_{1, \emptyset}^E$ corresponding to $n = 0$ appears,

\[
F_1^{K_X} = -F_1^{X/E} - \frac{1}{24}\log((-1)^{E\cdot E}\tilde{\mathcal{Q}}) + \sum_{d \geq 0}\tilde{\mathcal{Q}}^d \int_{\overline{\mathcal{M}}_{1,0}(E,d)}1
\]
The above stationary invariants are computed in \cite{Dji} and are given by,

\begin{align*}
    \sum_{d \geq 0}\tilde{\mathcal{Q}}^d \int_{\overline{\mathcal{M}}_{1,0}(E,d)}1  &= -\sum_{n \geq 1}\log(1- \tilde{\mathcal{Q}}^n) \\
    &= \sum_{n\geq 1}\sum_{k \geq 1}\frac{\tilde{\mathcal{Q}}^{nk}}{k} \\
    &= \sum_{n\geq 1}\left(\sum_{k|n, k \geq 1}\frac{1}{k}\right)\tilde{\mathcal{Q}}^n 
\end{align*}

Therefore, the $g=1$ log-local principle is,

\begin{align*}
    F_1^{K_X} &= -F_1^{X/E} - \frac{1}{24}\log\left((-1)^{E\cdot E}\tilde{\mathcal{Q}}\right) + \sum_{n\geq 1}\left(\sum_{k|n, k \geq 1}\frac{1}{k}\right)\tilde{\mathcal{Q}}^n \\
    &= -F_1^{X/E} - \frac{1}{24}\log (-1)^{E\cdot E}Q^E + \frac{1}{24}\sum_{\beta|\beta\cdot E > 0}(-1)^{\beta\cdot E+1}(\beta\cdot E)R_0(X(\log E),\beta)Q^{\beta} \\
    &+ \sum_{n \geq 1}\left(\sum_{j|n, j \geq 1}\frac{1}{j}\right)Q^{nE}\exp\left(n\sum_{\beta'|\beta'\cdot E > 0}(-1)^{\beta'\cdot E}(\beta'\cdot E)R_0(X(\log E),\beta')Q^{\beta'}\right)
\end{align*}
where we changed variables $\tilde{\mathcal{Q}} \leftrightarrow Q$ in the second equality. Let $[\bullet]_{Q^{\beta}}$ return the coefficient of $Q^{\beta}$ for an expression $\bullet$. We define the term,

\begin{equation}
    \label{eq:delta1}
    \delta_1(\beta) := \left[\sum_{n \geq 1}\left(\sum_{j|n, j \geq 1}\frac{1}{j}\right)Q^{nE}\exp\left[n\sum_{\beta'|\beta'\cdot E > 0}(-1)^{\beta'\cdot E}(\beta'\cdot E)R_0(X(\log E), \beta')Q^{\beta'}\right]\right]_{Q^{\beta}}
\end{equation}

Then, the individual Gromov-Witten invariants are related as,

\begin{theorem}[Genus-1 log-local]
\begin{align*}
    N_1(K_X, \beta) &= \frac{(-1)^{\beta\cdot E + 1}}{\beta\cdot E}R_{1, (\beta\cdot E)}(X, \beta) - \frac{1}{24}\left((-1)^{E\cdot E}-1\right)\log Q^E \\
    &+ \frac{1}{24}(-1)^{\beta\cdot E + 1}(\beta\cdot E)R_0(X(\log E), \beta) + \delta_1(\beta)
\end{align*}  
\end{theorem}

\begin{remark}
    The term $\delta_1(\beta)$ encapsulates contributions from the elliptic curve in the genus-1 log-local principle. For example, when $X = \FF_1$ given by the toric blow up $\pi: \FF_1 \rightarrow \bP^2$ with exceptional curve $C$, the values for $\delta_1(\pi^*dH-C)$ are $0,0,1,-35$ for $\beta = dH \in H_2(\bP^2, \mathbb{Z})$ for $d = 1,2,3,4$, respectively. 
\end{remark}

\subsection{Evaluation of Vertex $V_3$ in genus-1}

\label{sec:evaluation_vertexC}

We directly evaluate the genus-1 invariant associated to Vertex $V_3$ (\ref{eq:vertexC_invariant}) using a calculation from Appendix A, \cite{Bou}. The genus-1 invariant is,

    \[
    \int_{[\overline{\mathcal{M}}_{1,2}(\FF_1(\log F), \pi^*H)]^{vir}}-\lambda_1 ev_1^*([pt_1]) ev_2^*([pt_2])
    \]
    where $[pt_1] \in A^2(\FF_1)$ is the Poincar\'e dual of a point in the interior and $[pt_2] \in A^1(F)$ is the Poincar\'e dual of a point on a $\bP^1$-fiber class $F$ of $\FF_1.$ On $\overline{\mathcal{M}}_{1,1}$, we have,

    \[
    \lambda_1 = \frac{1}{12}\delta_0
    \]
    where $\delta_0 \in A^1(\overline{\mathcal{M}}_{1,1})$ is the class of a point. We take for a representative of $\delta_0$ the nodal rational cubic, and resolve the node. The invariant becomes,
    
    \[
    \left(\frac{1}{2}\right)\left(\frac{-1}{12}\right)\int_{[\overline{\mathcal{M}}_{0,4}(\FF_1(\log \bP^1), \pi^*H)]^{vir}}ev_1^*([pt_1]) ev_2^*([pt_2]) (ev_3\times ev_4)^*(D \times D)
    \]
where the $\frac{1}{2}$ comes from the two ways of labelling the two marked points that resolve the node. The class $D\times D$ is the diagonal curve class in $A^2(\FF_1 \times \FF_1)$, which is 

\[
D \times D = (1 \times pt) + (pt \times 1) + (\pi^*H \times \pi^*H) + (C \times C) 
\]
The first two terms in $D\times D$ contribute zero by the Fundamental Class Axiom. The last term also contributes zero by the Divisor Axiom, since $\pi^*H \cdot C = 0.$ Hence, the invariant becomes

\[
    \frac{-1}{24}\int_{[\overline{\mathcal{M}}_{0,4}(\FF_1(\log \bP^1), \pi^*H)]^{vir}}ev_1^*([pt_1]) ev_2^*([pt_2])ev_3^*(\pi^*H)ev_4^*(\pi^*H)
\]
By the Divisor Axiom, this is,

\[
\frac{-1}{24}\int_{[\overline{\mathcal{M}}_{0,2}(\FF_1(\log \bP^1), \pi^*H)]^{vir}}ev_1^*([pt_1]) ev_2^*([pt_2])
\]
The latter genus-0 invariant is the number of lines through two points, and hence the above evaluates to $\frac{-1}{24}$, which is also the coefficient of $\hbar^2$ in $(-i)(q^{\frac{1}{2}} - q^{\frac{-1}{2}})$ (see Proposition \ref{prop:vertexC}).

\bibliographystyle{alpha}
\bibliography{refs}
\end{document}